\documentclass[12pt]{scrartcl}
\usepackage[english]{babel}
\usepackage[utf8x]{inputenc}
\usepackage{amsfonts}
\usepackage{amsmath}
\usepackage{amssymb}
\usepackage{amsthm}
\usepackage{mathrsfs}
\usepackage[colorlinks=false,urlcolor=blue]{hyperref}
\usepackage{breakurl}
\usepackage{listings}
\usepackage{graphicx}
\usepackage{enumerate}
\usepackage{tikz}
\usetikzlibrary{patterns}

\addtokomafont{section}{\rmfamily}
\addtokomafont{subsection}{\rmfamily}
\usepackage{indentfirst}

\swapnumbers
\theoremstyle{plain}
\newtheorem{satz}{Theorem}[section]
\newtheorem{lem}[satz]{Lemma}

\newtheorem{prop}[satz]{Proposition}
\theoremstyle{definition}

\newtheorem{bem}[satz]{Remark}

\newcommand{\R}{\mathbb{R}}
\newcommand{\C}{\mathbb{C}}
\newcommand{\N}{\mathbb{N}}
\newcommand{\Z}{\mathbb{Z}}
\newcommand{\CP}{\mathbb{CP}}
\newcommand{\Id}{\operatorname{Id}}
\newcommand{\Ric}{\operatorname{Ric}}
\newcommand{\scal}{\operatorname{scal}}
\newcommand{\vol}{\operatorname{vol}}
\newcommand{\Vol}{\operatorname{Vol}}
\newcommand{\Sym}{\operatorname{Sym}}
\newcommand{\tr}{\operatorname{tr}}
\newcommand{\im}{\operatorname{im}}
\newcommand{\diag}{\operatorname{diag}}
\newcommand{\End}{\operatorname{End}}
\newcommand{\Aut}{\operatorname{Aut}}
\newcommand{\Hom}{\operatorname{Hom}}
\newcommand{\Ad}{\operatorname{Ad}}
\newcommand{\ad}{\operatorname{ad}}
\newcommand{\Cas}{\operatorname{Cas}}
\renewcommand{\S}{\mathscr{S}}
\newcommand{\Sl}{\mathfrak{S}}
\newcommand{\X}{\mathfrak{X}}
\newcommand{\U}{\operatorname{U}}
\newcommand{\SU}{\operatorname{SU}}
\newcommand{\su}{\mathfrak{su}}
\renewcommand{\u}{\mathfrak{u}}
\newcommand{\SO}{\operatorname{SO}}
\newcommand{\so}{\mathfrak{so}}
\newcommand{\Sp}{\operatorname{Sp}}
\newcommand{\spann}{\operatorname{span}}
\newcommand{\pr}{\operatorname{pr}}
\renewcommand{\i}{\mathrm{i}}
\renewcommand{\j}{\mathrm{j}}
\renewcommand{\H}{\mathcal{H}}
\newcommand{\m}{\mathfrak{m}}
\newcommand{\h}{\mathfrak{h}}
\newcommand{\g}{\mathfrak{g}}
\renewcommand{\k}{\mathfrak{k}}
\renewcommand{\t}{\mathfrak{t}}
\DeclareMathOperator*{\closedsum}{\overline{\bigoplus}}
\newcommand{\Einstein}{E}
\newcommand{\TT}{\S^2_{\mathrm{tt}}}
\newcommand{\hooklongrightarrow}{\lhook\joinrel\longrightarrow}
\newcommand{\Do}{\mathfrak{D}}
\newcommand{\CR}{\nabla^{\mathrm{red}}}
\newcommand{\Rcr}{R^{\mathrm{red}}}
\newcommand{\Lcr}{\Delta^{\mathrm{red}}}
\newcommand{\Eop}{\mathcal{E}}

\title{\rmfamily Coindex and rigidity of Einstein metrics on homogeneous Gray manifolds}
\author{Paul Schwahn*}
\date{\today}

\begin{document}

\maketitle
{\let\thefootnote\relax\footnotetext{*Institut für Geometrie und Topologie, Fachbereich Mathematik, Universität Stuttgart, Pfaffenwaldring 57, 70569 Stuttgart, Germany.}}

\begin{abstract}
\footnotesize
\begin{center}
\textbf{Abstract}
\end{center}

\noindent
Any $6$-dimensional strict nearly Kähler manifold is Einstein with positive scalar curvature. We compute the coindex of the metric with respect to the Einstein-Hilbert functional on each of the compact homogeneous examples. Moreover, we show that the infinitesimal Einstein deformations on $F_{1,2}=\SU(3)/T^2$ are not integrable into a curve of Einstein metrics.

\textit{MSC (2020):} 53C24, 53C25, 53C30.

\textit{Keywords:} Einstein metrics, Stability, Rigidity, Lichnerowicz Laplacian, nearly Kähler.
\end{abstract}

\section{Introduction}
\label{sec:intro}

The special case of dimension $6$ has been a primary focus of nearly Kähler geometry since P.-A. Nagy showed that every nearly Kähler manifold is locally isometric to a Riemannian product of $6$-dimensional nearly Kähler manifolds, nearly Kähler homogeneous spaces and twistor spaces over positive scalar curvature quaternionic-Kähler manifolds \cite{nagy}. Moreover, nearly Kähler manifolds that are non-Kähler (so-called \emph{strict} nearly Kähler manifolds) of dimension $6$ exhibit other notable properties, such as carrying a real Killing spinor and thus being Einstein with positive scalar curvature.

On a compact manifold $M$, Einstein metrics can be variationally characterized as critical points of the total scalar curvature functional $S$ (also called \emph{Einstein--Hilbert action}), defined on the set of all Riemannian metrics on $M$ of a fixed volume. Given a compact Einstein manifold $(M,g)$, one can ask whether $g$ locally maximizes $S$ (after restricting to a suitable subclass of Riemannian metrics). Such an Einstein metric $g$ is called \emph{stable} with respect to $S$. The linearized problem considers the Hessian $S_g''$ of the Einstein-Hilbert action at $g$. Accordingly, an Einstein metric $g$ is called \emph{linearly stable} if $S_g''\leq 0$ on the space of tt-tensors (i.e. trace- and divergence-free symmetric $2$-tensors on $M$). A closely related notion is that of \emph{infinitesimal deformability} of the Einstein metric $g$ -- it is called infinitesimally deformable if $S_g''$ is degenerate on tt-tensors.

A compact, $6$-dimensional, strict nearly Kähler manifold $(M,g,J)$ with scalar curvature normalized to $\scal_g=30$ (hence Einstein constant $\Einstein=5$) is called a \emph{Gray manifold}, after A. Gray, who studied them in the 70s. The stability and infinitesimal deformability of Einstein metrics on Gray manifolds have already been investigated. In \cite{nkstability}, U. Semmelmann, C. Wang and M. Y.-K. Wang show linear instability if the second or third Betti number does not vanish -- in fact, the coindex of $g$ (see Section \ref{sec:prelimst} for a definition) is bounded below by $b_2+b_3$. A. Moroianu and U. Semmelmann \cite{nearlykaehler} give a description of the space of infinitesimal Einstein deformations in terms of eigenspaces of the Hodge Laplacian on coclosed primitive $(1,1)$-forms. The present article generalizes this result to a similar description of eigenspaces of the Lichnerowicz Laplacian on tt-tensors to arbitrary eigenvalues not exceeding a certain threshold (see Lemma~\ref{epsilon}).

Homogeneous Gray manifolds have been classified by J.-B. Buitruille \cite{butruilleclass}. There are only four cases: $S^6=\frac{G_2}{\SU(3)}$, $S^3\times S^3=\frac{\SU(2)\times\SU(2)\times\SU(2)}{\Delta\SU(2)}$, $\CP^3=\frac{\Sp(2)}{\Sp(1)\U(1)}=\frac{\SO(5)}{\U(2)}$ and the flag manifold $F_{1,2}=\frac{\SU(3)}{T^2}$, all of them equipped with the Killing form metric (up to scaling). In \cite{WW}, C. Wang and M. Y.-K. Wang show instability of the latter three spaces. $S^6$ carries the round metric and is thus strictly stable.

One aim of this article is to improve the coindex estimates from \cite{nkstability} to equalities for the homogeneous examples. Our first main result can be stated as follows.

\begin{satz}\label{thm1}
Let $(M,g)$ be a homogeneous Gray manifold with standard metric $g$. The coindex of the Einstein metric $g$ is
\begin{itemize}
 \item equal to $2$ if $M=S^3\times S^3=\frac{\SU(2)\times\SU(2)\times\SU(2)}{\Delta\SU(2)}$,
 \item equal to $1$ if $M=\CP^3=\frac{\SO(5)}{\U(2)}$,
 \item equal to $2$ if $M=F_{1,2}=\frac{\SU(3)}{T^2}$.
\end{itemize}
\end{satz}

The destabilizing directions, i.e. contributions to the coindex, can be viewed as arising from harmonic $3$-forms in the first and from harmonic $2$-forms in the second and third case via the construction in \cite{nkstability}. For the last two cases, there is an additional geometric explanation: consider the Riemannian submersions given by the twistor fibrations
\begin{align*}
\CP^3=\frac{\SO(5)}{\U(2)}&\longrightarrow\frac{\SO(5)}{\SO(4)}=S^4,\\
F_{1,2}=\frac{\SU(3)}{T^2}&\longrightarrow\frac{\SU(3)}{\mathrm{S}(\SU(2)\U(1))}=\CP^2.
\end{align*}
In both cases, the canonical variation (scaling the base against the fiber) yields a destabilizing direction by \cite[Prop.~4.4]{WW}. For the flag manifold, there are actually three such fibrations whose canonical variations give rise to a two-dimensional space of tt-tensors, explaining the coindex of $2$ (see Remark~\ref{flagcs}).

Also worth noting is the $G$-invariant stability problem, in which the Einstein-Hilbert functional $S$ is restricted to the class of $G$-invariant metrics on a fixed homogeneous space $M=G/H$. Since the destabilizing directions on all three cases in Theorem~\ref{thm1} are $G$-invariant (as explained in Remark~\ref{Ginv}), it follows that their metrics are $G$-unstable. In fact, they are even $G$-strongly unstable, i.e. \emph{all} $G$-invariant tt-variations of the metric are destabilizing and hence these metrics are local minima of $S$ among $G$-invariant metrics. For $\CP^3$ \cite[Table~1, 7a]{ejlauret} and $F_{1,2}$ \cite[Table~2]{jlauret} this was already known from the results of E.~A.~Lauret and J.~Lauret. For $S^3\times S^3$ see Remark~\ref{s3xs3Gstr}.

Let us return to the general setting of a compact manifold $M$. An Einstein metric $g$ on $M$ is called \emph{rigid} if it is isolated in the moduli space of Einstein structures (disregarding variation by homothetic scaling and action of diffeomorphisms). If an Einstein manifold $(M,g)$ admits infinitesimal Einstein deformations, one naturally asks whether they are integrable into a curve of Einstein metrics on $M$. In fact, not every infinitesimally deformable Einstein must lie within a nontrivial curve of Einstein metrics. The first example of such a metric is the canonical symmetric metric on $\CP^1\times\CP^{2k}$ found by N. Koiso \cite{koisorigidity}, who started the investigation of stability and infinitesimal deformatibility of symmetric spaces \cite{koiso}. Another recent example due to W. Batat, S. J. Hall, T. Murphy and J. Waldron is the bi-invariant metric on $\SU(2n+1)$ \cite{sunrigidity}. We add one more example to this list by proving the following result.

\begin{satz}\label{thm2}
The Einstein metric on the Gray manifold $F_{1,2}$ is rigid, that is, its infinitesimal Einstein deformations are not integrable.
\end{satz}

In all of the above examples, integrability fails at an obstruction to second order (see the end of Section \ref{sec:prelimst}). We suspect that this phenomenon occurs generically. Given some infinitesimal Einstein deformation, i.e. an element of the null space of $S_g''$, the obstruction polynomial (\ref{obstruction}) has no immediate compulsion to vanish and should do so only coincidentally -- see for example the case $\SU(2n)$ in \cite{sunrigidity}.

Since Gray manifolds are Einstein, every infinitesimal deformation of the nearly Kähler structure corresponds to an infinitesimal Einstein deformation, but not necessarily vice versa \cite{nearlykaehler}. Infinitesimal deformability of the nearly Kähler structure has been investigated by A. Moroianu, P.-A. Nagy and U. Semmelmann \cite{nkdeformations}. The question whether a given infinitesimal nearly Kähler deformation can be integrated into a curve of nearly Kähler structures has been studied by L. Foscolo in \cite{foscolo}, where a similar polynomial occurs as integrability obstruction to second order. In particular, he showed that the infinitesimal nearly Kähler deformations on $F_{1,2}$ are all obstructed. One can view Theorem~\ref{thm2} as a generalization of this result to the Einstein picture. The Einstein metrics and thus nearly Kähler structures on homogeneous Gray manifolds other than $F_{1,2}$ are automatically rigid since they possess no infinitesimal deformations \cite{nearlykaehler}.

This article is organized as follows. In Section \ref{sec:prelim}, notation is fixed and the necessary preliminaries are recapitulated. Section \ref{sec:general} concerns itself with a description of eigenspaces of the Lichnerowicz Laplacian on tt-tensors on general Gray manifolds as well as a discussion of the homogeneous case, in which explicit calculations are possible by means of harmonic analysis. These results are applied in Section \ref{sec:examples} to each of the unstable Gray manifolds $S^3\times S^3$, $\CP^3$ and $F_{1,2}$ to obtain the results collected in Theorem~\ref{thm1}. Finally, Section \ref{sec:rigidity} recalls the description of the infinitesimal Einstein deformations on $F_{1,2}$ given in \cite{hermitianlaplace} and proceeds to show the nonintegrability to second order, proving Theorem~\ref{thm2}.

The author owes gratitude to Prof. U. Semmelmann for helpful exchanges about a gap in the argument given in the proof of \cite[Thm.~5.1]{nearlykaehler} (the corrected argument is the proof of Lemma~\ref{equivsystem}, which includes the aforementioned as the special case $\lambda=10$). Furthermore, the author would like to thank Prof. G. Weingart for his useful suggestions regarding the rigidity argument.

\section{Preliminaries}
\label{sec:prelim}

\subsection{Nearly Kähler manifolds}
\label{sec:prelimnk}

An almost Hermitian manifold $(M,g,J)$ is an even-dimensional Riemannian manifold $(M,g)$ with an almost complex structure $J$ that is compatible with the metric, i.e.
\[g(JX,JY)=g(X,Y)\]
for any $X,Y\in T_pM$. The Kähler form $\omega$ is then defined by
\[\omega(X,Y):=g(JX,Y).\]
Any almost Hermitian structure has an associated \emph{canonical Hermitian connection} $\bar\nabla$ (see, for example, \cite[Sec.~2]{3sym} for a general definition). In particular, it satisfies $\bar\nabla g=0$ and $\bar\nabla J=0$.

Let $\nabla$ denote the Levi-Civita connection of the Riemannian manifold $(M,g)$. An almost Hermitian manifold $(M,g,J)$ is called \emph{nearly Kähler} if $\nabla J$ is skew-symmetric, or equivalently, if
\[(\nabla_XJ)X=0\]
for all $X\in T_pM$. In this case, the canonical Hermitian connection can be described by
\[\bar\nabla_XY=\nabla_XY-\frac{1}{2}J(\nabla_XJ)Y\]
for any two vector fields $X,Y\in\X(M)$. A nearly Kähler manifold is called \emph{strictly nearly Kähler} if it is not Kähler. \emph{Gray manifolds} are compact strict nearly Kähler manifolds of dimension $6$.

As usual, the almost complex structure $J$ defines a splitting of the complexified cotangent bundle $T^\ast M^\C=\Lambda^{1,0}M\oplus\Lambda^{0,1}M$ and hence of the bundle of $k$-forms into $(p,q)$-forms with $p+q=k$. The complex bundle of $(p,q)$-forms will be denoted with the prefix $\Lambda^{p,q}$, and the space of its smooth sections by $\Omega^{p,q}$. The Kähler form $\omega$ is of type $(1,1)$. A $(p,q)$-form $\alpha$ is called \emph{primitive} if it vanishes under contraction with the Kähler form, i.e. if $\omega\lrcorner\alpha=0$. We will denote the bundle of primitive $(p,q)$-forms by $\Lambda^{p,q}_0$. Furthermore, let $\Lambda^{p,q}_\R$ denote the projection of the complex bundle $\Lambda^{p,q}$ to the real bundle $\Lambda^{p+q}$.

Likewise, the bundle $\Sym TM$ of $g$-symmetric endomorphisms of the tangent bundle splits into a direct sum $\Sym^+TM\oplus\Sym^-TM$, where the elements of $\Sym^\pm TM$ commute (resp. anticommute) with $J$. We further denote by $\Sym^+_0TM$ the subbundle of trace-free endomorphisms in $\Sym^+TM$, and with $\S^\pm$, $\S^+_0$ the spaces of smooth sections in the respective bundles.

Let $\S^k=\Gamma(\Sym^kT^\ast M)$ denote the space of symmetric $k$-tensor fields. Note that the metric yields a natural identification $\Sym^2T^\ast M\cong\Sym TM$. The subspace of tt-tensors in $\S^2$ (i.e. $h\in\S^2$ satisfying $\tr_gh=0$ and $\delta h=0$) will be denoted by $\TT$.

If $(M,g,J)$ is nearly Kähler, then the tensor $\Psi^+:=\nabla\omega$ is to\-tal\-ly skew-sym\-met\-ric and in fact the real part of a $\bar\nabla$-parallel complex volume form $\Psi^++\i\Psi^-$. The imaginary part $\Psi^-$ can be described by $X\lrcorner\Psi^-=J\circ(\nabla_XJ)$ for all $X\in TM$. The strict nearly Kähler case is characterized by the non-vanishing of $\Psi^+$.

Let $(M,g,J)$ be a strict nearly Kähler manifold of dimension $6$. There are $\bar\nabla$-parallel isomorphisms
\begin{align}
 TM&\cong\Lambda^{2,0}_\R M&\Sym^+_0TM&\cong\Lambda^{1,1}_{0,\R}M&\Sym^-TM&\cong\Lambda^{2,1}_\R M\label{iso}\\
 X&\mapsto X\lrcorner\Psi^+&h&\mapsto J\circ h&h&\mapsto h_\ast\Psi^+\notag
\end{align}
of vector bundles with structure group $\SU(3)$, each arising from an equivalence of $\SU(3)$-representations. Here, $h_\ast$ denotes the extension of the endomorphism $h\in\End TM$ to tensor bundles as a derivation.

\subsection{Stability and rigidity}
\label{sec:prelimst}

The \emph{Lichnerowicz Laplacian} $\Delta_L$ of a Riemannian manifold $(M,g)$ is an operator that generalizes the Hodge Laplacian $\Delta$ on differential forms to tensor fields of any rank. It is defined by
\[\Delta_L:=\nabla^\ast\nabla+q(R),\]
where $q(R)$ is the curvature endomorphism acting on tensors by
\[q(R):=\sum_{i<j}(e_i\wedge e_j)_\ast R(e_i,e_j)\]
for some local orthonormal frame $(e_i)$ of $TM$. The asterisk denotes the natural action of $\Lambda^2T\cong\so(T)$. In particular, $q(R)=\Ric$ on $1$-forms.

On an almost Hermitian manifold, we analogously define the \emph{Hermitian Laplace operator} $\bar\Delta$ by replacing the Levi-Civita connection $\nabla$ in the above definition by the canonical Hermitian connection $\bar\nabla$, i.e.
\[\bar\Delta:=\bar\nabla^\ast\bar\nabla+q(\bar R).\]
Here, $R$ and $\bar R$ denote the curvature tensors of the connections $\nabla$ and $\bar\nabla$, respectively. Both $\Delta_L$ and $\bar\Delta$ are instances of the \emph{standard Laplacian} of a given connection (see \cite{standardlapl}), an operator with several neat properties -- for example, it commutes with parallel bundle maps. Comparison formulas for the two Laplace operators in the setting of $6$-dimensional nearly Kähler manifolds can be found in \cite{hermitianlaplace} and \cite{nearlykaehler}. For our purposes, it is important to note that $\Delta$ and $\bar\Delta$ coincide on coclosed primitive $(1,1)$-forms, as well as on coclosed $(2,1)$- and $(1,2)$-forms, which follows from combining Cor.~3.5 and Cor.~4.4 of \cite{nearlykaehler}.

Consider a fixed compact orientable smooth manifold $M$ of dimension $n>2$. On the set of all Riemannian metrics on $M$, the \emph{total scalar curvature functional} (or \emph{Einstein--Hilbert action}) is defined by
\[g\mapsto S(g)=\int_M\scal_g\vol_g.\]
Einstein metrics on $M$ are then precisely the critical points of the restriction of $S$ to metrics of a fixed total volume. Let $(M,g)$ be an Einstein manifold with $\Ric=\Einstein g$. If $(M,g)$ not isometric to the standard sphere, there is a well-known decomposition
\[\S^2=\R g\oplus C^\infty_gg\oplus L_\X g\oplus\TT\]
that is orthogonal with respect to the second variation $S''_g$ (see \cite{besse}). Furthermore,
\begin{align*}
S''_g&>0\text{ on }C^\infty_gg,\text{ where }&C^\infty_g&=\{f\in C^\infty(M)\,|\,(f,\mathbf{1})_{L^2}=0\},\\
S''_g&=0\text{ on }&L_\X g&=\{L_Xg\,|\,X\in\X(M)\},\\
S''_g(h,h)&=-\dfrac{1}{2}\left(\Delta_L h-2\Einstein h,h\right)_{L^2}\text{ on }&\TT&=\{h\in\S^2\,|\,\tr_gh=0,\ \delta h=0\}.
\end{align*}
On the latter space, $S''_g$ has finite coindex and nullity, i.e. the maximal subspace of $\TT$ on which $S''_g$ is nonnegative is finite-dimensional. The sum $L_\X g\oplus\TT=T_g\Sl$ can also be regarded as formal tangent space to the set $\Sl$ of metrics with constant scalar curvature and fixed total volume.

The stability problem is to decide whether an Einstein metric $g$ is a local maximum or a saddle point of $S\big|_{\Sl}$. We are primarily concerned with the linearized version, considering only the second variation of $S$ at $g$. An Einstein metric $g$ is called \emph{(linearly) stable} if $S''_g\big|_{\TT}\leq0$, or, equivalently, if $\Delta_L\geq 2\Einstein$ on $\TT$. If strict inequality holds, we call $g$ \emph{strictly stable}. On the other hand, $g$ is called \emph{(linearly) unstable} if there exists $h\in\TT$ such that $S_g''(h,h)>0$, or, equivalently, if $(\Delta_Lh,h)_{L^2}<2\Einstein\|h\|^2_{L^2}$. The dimension of the maximal subspace of $\TT$ on which $S_g''>0$ is called the \emph{coindex} of $g$.

A closely related notion is that of rigidity. An Einstein metric $g$ is called \emph{rigid} if it is isolated in the moduli space, i.e. the space of Einstein metrics modulo diffeomorphisms and homotheties. Since the moduli space is locally arcwise connected \cite[Cor.~12.52]{besse}, rigidity of $g$ is equivalent to the nonexistence of a smooth curve $(g_t)$ of Einstein metrics through $g=g_0$ with nonvanishing first-order jet $\dot g_0\in\TT$.

Denote by $\varepsilon(g)=\{h\in\TT\,|\,\Delta_Lh=2\Einstein h\}$ the null space of $S''_g$, also called the space of \emph{infinitesimal Einstein deformations (IED)}. If $\varepsilon(g)\neq0$, we call $g$ \emph{infinitesimally deformable}. A metric with $\varepsilon(g)=0$ is automatically rigid \cite[Cor.~12.66]{besse} -- in particular, strict stability implies rigidity.

In general, IED need not be integrable into a curve of Einstein metrics. On the set of unit volume Riemannian metrics, define the Einstein operator $\Eop$ by
\[\Eop(g):=\Ric_g-\frac{S(g)}{n}g.\]
Then a metric $g$ is Einstein if and only if $\Eop(g)=0$. An IED $h\in\varepsilon(g)$ is called \emph{formally integrable to order $k$} if there exist $h_2,\ldots,h_k\in\S^2$ such that
\[\Eop\left(g+th+\sum_{j=2}^k\frac{t^k}{k!}h_k\right)=0.\]
A classical result \cite[Cor.~12.50]{besse} is that an IED $h\in\varepsilon(g)$ can be integrated into a curve $(g_t)$ of Einstein metrics with $\dot g_0=h$ if and only if it is formally integrable to all orders $k\geq2$.

The integrability criterion to each order can be expressed in terms of derivatives of $\Eop$. By a result of N. Koiso \cite[Lem.~4.7]{koisorigidity}, $h\in\varepsilon(g)$ is integrable to order $2$ if and only if $\Eop''_g(h,h)\perp\varepsilon(g)$ in the $L^2$ sense. Also due to N. Koiso \cite[Lem.~4.3]{koisorigidity} is the formula
\begin{align}
2\left(\Eop''_g(h,h),h\right)_{L^2}=\int_M\big(&2\Einstein h_{ij}h_{ik}h_{jk}+3(\nabla_{e_i}\nabla_{e_j}h)_{kl}h_{ij}h_{kl}\notag\\
&-6(\nabla_{e_i}\nabla_{e_j}h)_{kl}h_{ik}h_{jl}\big)\vol_g
\label{obstruction}
\end{align}
for the second order obstruction, where we implicitly sum over a local orthonormal frame $(e_i)$ of $TM$. The vanishing of the quantity in (\ref{obstruction}) is a necessary condition for the integrability of $h$.

\subsection{Harmonic analysis}
\label{sec:prelimha}

Let $(M=G/H,g)$ be a Riemannian homogeneous space, where $G$ is some Lie group and $H$ is a closed subgroup. We will always denote the corresponding Lie algebras by $\g$ and $\h$, respectively. The homogeneous space $M$ is called \emph{reductive} if there exists an $\Ad\big|_H$-invariant complement $\m$ of $\h\subset\g$. This is always the case if $H$ is compact (in particular, if $G$ is compact). Through the canonical projection $\pi: G\to M$, the reductive complement $\m\subset\g\cong T_eG$ is canonically identified with the tangent space $T_oM$ at the base point $o=eH$.

The $G$-invariant metric on a reductive Riemannian homogeneous space $(M,g)$ is determined by an $\Ad(H)$-invariant inner product on $\m$. Suppose that $Q$ is an $\Ad(G)$-invariant inner product on $\g$. Then $\m:=\h^\perp$ is an $\Ad(H)$-invariant subspace. We call $(M,g)$ a \emph{normal} homogeneous space if the metric is induced by the restriction $Q$ to $\m$, i.e. $g_o=Q\big|_{\m\times\m}$. If $G$ is compact and semisimple, then the Killing form $B_\g$ is negative-definite. In this case, the \emph{standard} metric is defined by $g_o=-B_\g\big|_{\m\times\m}$.

A normal homogeneous space is in particular \emph{naturally reductive}, i.e.
\[g_o([X,Y]_\m,Z)+g_o(Y,[X,Z]_\m)=0\]
(where $X_\m$ denotes the projection of $X$ to $\m$) holds for all $X,Y,Z\in\m$.

Let $\rho: H\to\Aut V$ be a finite-dimensional (real or complex) representation. Denote by $VM=G\times_\rho V$ the associated homogeneous vector bundle over $M$. Its sections can be viewed as $H$-equivariant smooth $V$-valued functions on $G$ -- the isomorphism is explicitly given by
\[\Gamma(VM)\stackrel{\cong}{\longrightarrow} C^\infty(G,V)^H:\ s\mapsto\hat s,\]
where $s(xH)=[x,\hat s(x)]$ for any $x\in G$. Left-translation on sections of $VM$ gives rise to the \emph{left-regular representation} on $C^\infty(G,V)^H$, explicitly given by
\[\ell:\ G\to\Aut C^\infty(G,V)^H:\ (\ell(x)f)(y)=f(x^{-1}y)\]
for $x,y\in G$.

If $M$ is reductive, we can write every tensor bundle as an associated bundle of some tensor power of the reductive complement. For example,
\[\S^2=\Gamma(\Sym^2T^\ast M)\cong\Gamma(G\times_\rho\Sym^2\m)\cong C^\infty(G,\Sym^2\m)^H\]
(note that the $H$-representations $\m$ and $\m^\ast$ are equivalent via the Riemannian metric).

For a compact Lie group $G$, denote by $\hat{G}$ the set of dominant integral weights of $G$ (after choosing a suitable maximal torus $T\subset G$). Recall that the elements of $\hat G$ are in one-to-one correspondence with equivalence classes of irreducible complex representations of $G$. Any representative of such a class with highest weight $\gamma\in\hat G$ will be denoted by $(V_\gamma,\rho_\gamma)$. Let $V$ be a unitary representation of $H$. The homogeneous version of the Peter-Weyl theorem \cite[Thm.~5.3.6]{wallach} states that the left-regular representation decomposes into
\begin{equation}
L^2(G,V)^H\cong\closedsum_{\gamma\in\hat G}V_\gamma\otimes\Hom_H(V_\gamma,V).\label{peterweyl}
\end{equation}
Here, $\Hom_H(V_\gamma,V)$ simply counts the multiplicity of $V_\gamma$ inside $L^2(G,V)^H$ and is called the space of \emph{Fourier (matrix) coefficients}. The equivalence in (\ref{peterweyl}) is made explicit by
\begin{equation}
V_\gamma\otimes\Hom_H(V_\gamma,V)\hookrightarrow C^\infty(G,V)^H:\ v\otimes F\mapsto\left(x\mapsto F(\rho_\gamma^{-1}(x)v)\right).\label{pwembedding}
\end{equation}

Let $V,W$ be unitary representations of $H$ and $\Do: \Gamma(VM)\to\Gamma(WM)$ be a $G$-invariant differential operator. Combining (\ref{peterweyl}) with Schur's Lemma, the operator $\Do$ acts as a linear mapping
\[\Do:\ \Hom_H(V_\gamma,V)\longrightarrow\Hom_H(V_\gamma,W)\]
for each fixed $\gamma\in\hat G$. We call this mapping the \emph{prototypical differential operator} associated to $\Do$ and $\gamma$ (as introduced by U. Semmelmann and G. Weingart in \cite{sw}).

On a reductive homogeneous space, a choice of reductive complement $\m$ determines a $G$-invariant connection $\CR$ on $VM$, called the \emph{canonical reductive} (or \emph{Ambrose-Singer}) connection, by stipulating that
\begin{equation}
\widehat{\CR_Xs}=\tilde{X}(\hat s)\label{CR}
\end{equation}
for all $X\in TM$, $s\in\Gamma(VM)$, where the \emph{horizontal lift} $\tilde X\in TG$ is the unique vector in the \emph{canonical horizontal distribution} $\H=\bigcup_{x\in G}dl_x(\m)$ such that $d\pi(\tilde X)=X$. This connection has the important property that all $G$-invariant sections of $VM$ are parallel. If $(M,g)$ is naturally reductive, $\CR$ is a metric connection with parallel totally skew torsion tensor $\tau$, given (at the base point) by
\[\tau_o(X,Y)=-[X,Y]_\m.\]

On any representation $\rho: G\to\Aut V$ of a compact Lie group $G$, the \emph{Casimir operator} with respect to a fixed $\Ad(G)$-invariant inner product on $\g$ is the equivariant endomorphism of $V$ defined by
\[\Cas^{\g,Q}_\rho=-\sum_i\rho_\ast(e_i)^2,\]
where $(e_i)$ is an orthonormal basis of $\g$ with respect to $Q$. We omit the superscript $Q$ if the choice of inner product is clear from context. For $\gamma\in\hat G$, the Casimir operator on $V_\gamma$ acts as multiplication with the \emph{Casimir constant}
\begin{equation}
\Cas^{\g,Q}_\gamma=\langle\gamma,\gamma+2\delta_\g\rangle_{\t^\ast,Q}\label{freudenthal}
\end{equation}
by Freudenthal's formula, cf. \cite{FH}. Here, $\langle\cdot,\cdot\rangle_{\t^\ast,Q}$ is the inner product induced by $Q$ on the dual $\t^\ast$ of the Lie algebra $\t$ of the torus $T\subset G$, while $\delta_\g$ denotes the half-sum of positive roots of $\g$.

A crucial fact \cite[Lem.~5.2]{hermitianlaplace} is that on a normal homogeneous space with Riemannian metric induced by an $\Ad(G)$-invariant inner product $Q$ on $\g$, the standard Laplacian of $\CR$ is precisely the Casimir operator of $G$ acting on the left-regular representation, i.e.
\begin{equation}
\Lcr:=(\CR)^\ast\CR+q(\Rcr)=\Cas^{\g,Q}_\ell.\label{laplcas}
\end{equation}
In particular, the prototypical differential operator associated to $\Lcr$ and $\gamma$ is simply multiplication by the Casimir constant. In other words, the eigenspaces of $\Lcr$ are the isotypical components 
\[V_\gamma\otimes\Hom_H(V_\gamma,V)\]
in the Peter-Weyl decomposition (\ref{peterweyl}). The eigenvalues are readily computable by means of Freudenthal's formula (\ref{freudenthal}).

It should be noted that in the symmetric case, the torsion of $\CR$ vanishes. Hence $\CR$ coincides with the Levi-Civita connection $\nabla$. It follows that $\Delta_L=\Lcr$, so the spectrum of the Lichnerowicz Laplacian on any tensor bundle is easily computable, facilitating the foundational work by N. Koiso on the stability of symmetric spaces \cite{koiso}.

\subsection{3-symmetric spaces}
\label{sec:prelim3s}

A homogeneous space $M=G/H$ is called \emph{3-symmetric} if there exists an automorphism $\sigma\in\Aut G$ of order $3$ such that $G_0^\sigma\subset H\subset G^\sigma$, where $G^\sigma$ is the fixed point set of $\sigma$ and $G^\sigma_0$ is the connected component of the identity in $G^\sigma$.

The complexified Lie algebra $\g^\C$ decomposes into eigenspaces of the differential at the base point $\sigma_\ast: \g\to\g$ as
\[\g^\C=\h^\C\oplus\m^+\oplus\m^-.\]
The eigenvalues of $\sigma_\ast$ are $1$ on $\h^\C$, $\j:=e^{\frac{2\pi\i}{3}}$ on $\m^+$ and $\j^2=\bar\j=e^{\frac{4\pi\i}{3}}$ on $\m^-$, respectively. $M$ then carries a natural $G$-invariant almost complex structure $J$ with $\pm\i$-eigenspaces $\m^\pm$, given by
\[\sigma_\ast\big|_\m=\frac{1}{2}\Id_\m+\frac{\sqrt{3}}{2}J_o\]
at the base point, where $\m^\C=\m^+\oplus\m^-$. Furthermore, $M$ is a reductive homogeneous space, since $\m$ is invariant under the adjoint action of $H\subset G^\sigma$.

When endowed with a $G$-invariant Riemannian metric $g$ compatible with $J$, $(M,g)$ is called a \emph{Riemannian 3-symmetric space}. In particular, $(M,g,J)$ is almost Hermitian. Furthermore, the almost Hermitian structure $(M,g,J)$ is nearly Kähler if and only if $(M,g)$ is naturally reductive \cite[Prop.~3.8]{3sym}.

For an extensive treatment of 3-symmetric spaces, see \cite{3sym}. The final thing we need for our purposes is the key observation \cite[Prop.~3.5]{3sym} that on a Riemannian 3-symmetric space, the canonical reductive connection $\CR$ associated to $\m$ coincides with the canonical Hermitian connection $\bar\nabla$ defined in Section \ref{sec:prelimnk}.

\section{Small Lichnerowicz eigenvalues on Gray manifolds}
\label{sec:general}

Throughout what follows, let $(M,g,J)$ be a Gray manifold. In order to classify destabilizing directions for the Einstein-Hilbert functional, we need to find all tt-eigentensors of the Lichnerowicz Laplacian to eigenvalues smaller than the critical eigenvalue $2\Einstein$. That is, we want to solve the system
\begin{equation}
\begin{cases}
 \Delta_Lh=\lambda h,\\
 \delta h=0
\end{cases}\tag{L1}\label{laplace1}
\end{equation}
in $h\in\S^2_0$ for some $\lambda<2\Einstein=10$. We follow the discussion in \cite[Sec.~5]{nearlykaehler} to transform (\ref{laplace1}) into an eigenvalue problem for the more familiar Hodge-deRham Laplacian. Viewed as a section of $\Sym TM$, the tensor $h$ splits into $h=h^++h^-$ with $h^+\in\S^+_0$ and $h^-\in\S^-$. By applying the bundle isomorphisms given in (\ref{iso}), we obtain tensors $\varphi:=h^+\circ J\in\Omega^{1,1}_{0,\R}$ and $\sigma:=h^-_\ast\Psi^+\in\Omega^{2,1}_{0,\R}$ carrying the information of $h$.

\begin{lem}\label{equivsystem}
Under the isomorphisms $h^+\mapsto\varphi$ and $h^-\mapsto\sigma$ above, if $\lambda<16$, the system of equations {\upshape(\ref{laplace1})} is equivalent to
\begin{equation}
\begin{cases}
\Delta\varphi=(\lambda-6)\varphi-\delta\sigma,\\
\Delta\sigma=(\lambda-4)\sigma-4d\varphi,\\
\delta\varphi=0,\\
\delta\sigma\in\Omega^{(1,1)}_{0,\R}.
\end{cases}\tag{L2}\label{laplacesystem}
\end{equation}
\end{lem}

\begin{proof}
Using the formulae from Prop.~3.4 and Cor.~4.4 of \cite{nearlykaehler}, the first equation of (\ref{laplace1}) can be rewritten as
\begin{align*}
(\bar\nabla^\ast\bar\nabla+q(\bar R))(h^++h^-)=&\,\lambda(h^++h^-)-(3h^++s)-(2h^--(\delta h^-\lrcorner\Psi^++\delta\sigma)\circ J)\\
&-3h^+-2h^-,
\end{align*}
with $s\in\S^-$ defined by $s_\ast\Psi^+=2\delta h^+\wedge\omega+4d\varphi$. We note that $(\delta h^-\lrcorner\Psi^++\delta\sigma)\circ J$ is necessarily a traceless symmetric $2$-tensor, hence automatically $\delta h^-\lrcorner\Psi^++\delta\sigma\in\Omega^{1,1}_{0,\R}$, or, equivalently,
\[\delta h^-\lrcorner\Psi^++(\delta\sigma)_{2,0}=0.\]
Using that $\bar\nabla^\ast\bar\nabla+q(\bar R)$ preserves the spaces $\S^\pm$ and $\Omega^{2,0}_\R$, we can write (\ref{laplace1}) equivalently as
\[\begin{cases}
   (\bar\nabla^\ast\bar\nabla+q(\bar R))h^+=(\lambda-6)h^++(\delta\sigma)_{1,1}\circ J,\\
   (\bar\nabla^\ast\bar\nabla+q(\bar R))h^-=(\lambda-4)h^--s,\\
   \delta h^++\delta h^-=0.
  \end{cases}\]
Let now $\eta\in\Omega^1$ such that $(\delta\sigma)_{2,0}=\eta\lrcorner\Psi^+$. Since $\delta h^+=-J\delta\varphi$ and $\bar\nabla^\ast\bar\nabla+q(\bar R)$ commutes with the bundle isomorphisms from (\ref{iso}), we can apply them to obtain
\[\begin{cases}
   (\bar\nabla^\ast\bar\nabla+q(\bar R))\varphi=(\lambda-6)\varphi-(\delta\sigma)_{1,1},\\
   (\bar\nabla^\ast\bar\nabla+q(\bar R))\sigma=(\lambda-4)\sigma-2\eta\wedge\omega-4d\varphi,\\
   \delta\varphi=J\eta,\\
   (\delta\sigma)_{2,0}=\eta\lrcorner\Psi^+.
  \end{cases}\]
Using the remaining formulae in Cor.~3.5 and Cor.~4.4 of \cite{nearlykaehler}, we see that this is equivalent to
\[\begin{cases}
    \Delta\varphi=(\lambda-6)\varphi-\delta\sigma,\\
    \Delta\sigma=(\lambda-4)\sigma-4d\varphi-4\eta\wedge\omega,\\
   \delta\varphi=J\eta,\\
   (\delta\sigma)_{2,0}=\eta\lrcorner\Psi^+.
\end{cases}\]
Suppose now that $\lambda<16$. By applying $\delta$ to the first line of the above, it follows that
\[\Delta\delta\varphi=\delta\Delta\varphi=(\lambda-6)\delta\varphi.\]
The Lichnerowicz estimate $\Delta\geq 2q(R)=2\Einstein=10$ on coclosed $1$-forms now implies that $\delta\varphi=0$ and hence $\eta=0$, simplifying the above to (\ref{laplacesystem}).\footnote{This bridges the gap in the proof of \cite[Thm.~5.1]{nearlykaehler}, where $\eta=0$ was assumed without justification.}
\end{proof}

The space of solutions is described by the following lemma, which is a generalization of \cite[Lem.~5.2]{nearlykaehler}.

\begin{lem}\label{epsilon}
Suppose that $\lambda=10-\varepsilon$ in system {\upshape (\ref{laplacesystem})} for some $\varepsilon>0$. Denote with $E(\mu):=\ker(\Delta-\mu)\big|_{\Omega^{(1,1)}_{0,\R}}\cap\ker\delta$ the $\mu$-eigenspace of $\Delta$ on coclosed primitive $(1,1)$-forms.
\begin{enumerate}[\upshape (i)]
 \item Suppose that $\varepsilon<\frac{25}{4}$ and $\varepsilon\neq6$. Then the space of solutions to system {\upshape (\ref{laplacesystem})} is isomorphic to the direct sum $E(\mu_1)\oplus E(\mu_2)\oplus E(\mu_3)$ where $\mu_{1,2}=7-\varepsilon\pm\sqrt{25-4\varepsilon}$ and $\mu_3=6-\varepsilon$. The isomorphism is given by
\[\Psi:\ (\varphi,\sigma)\mapsto((3-\sqrt{25-4\varepsilon})\varphi+\delta\sigma,(3+\sqrt{25-4\varepsilon})\varphi+\delta\sigma,\ast d\sigma).\]
The inverse is given by
\[\Phi:\ (\alpha,\beta,\gamma)\mapsto\left(\frac{\beta-\alpha}{2\sqrt{25-4\varepsilon}},\frac{d\beta-d\alpha}{2(6-\varepsilon)\sqrt{25-4\varepsilon}}+\frac{d\alpha+d\beta}{2(6-\varepsilon)}-\frac{\ast d\gamma}{6-\varepsilon}\right).\]
If $6<\varepsilon<\frac{25}{4}$, then $E(\mu_3)$ becomes trivial and thus $\gamma=\ast d\sigma=0$.
 \item If $\varepsilon=6$, then the space of solutions to {\upshape (\ref{laplacesystem})} is isomorphic to $E(2)\oplus\ker\Delta\big|_{\Omega^3}$, with isomorphism given by
\[\Psi_6:\ (\varphi,\sigma+\tau)\mapsto(\varphi,\tau)=\left(-\frac{1}{4}\delta\sigma,\tau\right)\]
for any $\sigma\in\im\Delta\big|_{\Omega^3}$ and $\tau\in\ker\Delta\big|_{\Omega^3}$, and inverse
\[\Phi_6:\ (\varphi,\tau)\mapsto\left(\varphi,-2d\varphi+\tau\right).\]
 \item If $\varepsilon=\frac{25}{4}$, then the space of solutions to {\upshape (\ref{laplacesystem})} is isomorphic to $E(\frac{3}{4})$, with isomorphism given by
\[\Psi_\frac{25}{4}:\ (\varphi,\sigma)\mapsto \varphi=-\frac{1}{3}\delta\sigma\]
and inverse
\[\Phi_\frac{25}{4}:\ \varphi\mapsto\left(\varphi,-4d\varphi\right).\]
 \item If $\varepsilon>\frac{25}{4}$, then the space of solutions to {\upshape (\ref{laplacesystem})} is trivial.
\end{enumerate}
\end{lem}
\begin{proof}
The proof of the first part works completely analogously to the one of \cite[Lem.~5.2]{nearlykaehler}. We observe that if $(\varphi,\sigma)$ is a solution to (\ref{laplacesystem}), then
\[\begin{pmatrix}
   \Delta\varphi\\
   \Delta\delta\sigma
  \end{pmatrix}=A\begin{pmatrix}
   \varphi\\
   \delta\sigma
  \end{pmatrix},\quad A:=\begin{pmatrix}
   4-\varepsilon&-1\\
   -4(4-\varepsilon)&10-\varepsilon
  \end{pmatrix}.
\]
The eigenvalues of the matrix $A$ are $\mu_{1,2}=7-\varepsilon\pm\sqrt{25-4\varepsilon}$ with corresponding eigenvectors $v_{1,2}=(\begin{smallmatrix}3\mp\sqrt{25-4\varepsilon}\\1\end{smallmatrix})$, if $\varepsilon\neq\frac{25}{4}$.

Let $(\alpha,\beta,\gamma):=\Psi(\varphi,\sigma)$. Then
\begin{align*}
\begin{pmatrix}\alpha\\\beta\end{pmatrix}&=\begin{pmatrix}3-\sqrt{25-4\varepsilon}&1\\3+\sqrt{25-4\varepsilon}&1\end{pmatrix}\begin{pmatrix}\varphi\\\delta\sigma\end{pmatrix},\\
\begin{pmatrix}\varphi\\\delta\sigma\end{pmatrix}&=\frac{1}{2\sqrt{25-4\varepsilon}}\begin{pmatrix}-1&1\\3+\sqrt{25-4\varepsilon}&-3+\sqrt{25-4\varepsilon}\end{pmatrix}\begin{pmatrix}\alpha\\\beta\end{pmatrix}.
\end{align*}
If $\varepsilon<\frac{25}{4}$, then $25-4\varepsilon>0$ and we can always recover the data $(\varphi,\delta\sigma)$ from $(\alpha,\beta,\gamma)$. We also have
\[\ast d\gamma=\ast d\ast d\sigma=-\delta d\sigma\]
and thus
\[d\delta\sigma-\ast d\gamma+4d\varphi=\Delta\sigma+4d\varphi=(6-\varepsilon)\sigma.\]
Hence if $\varepsilon\neq6$, we can also recover $\sigma$. In total, $\Psi$ is invertible, and one can check that its inverse is given by $\Phi$.

If $6<\varepsilon<\frac{25}{4}$, then $\mu_3<0$ and since $\Delta$ is nonnegative, $E(\mu_3)$=0.

Let $\varepsilon=6$. If $(\varphi,\sigma)$ is a solution to (\ref{laplacesystem}), then
\[\Delta d\sigma=d\Delta\sigma=-4d\Delta d\varphi=-4\Delta d^2\varphi=0,\]
i.e. $d\sigma$ is harmonic. But this implies that $\delta d\sigma=0$ and hence $d\sigma=0$. Also, $2\varphi+\delta\sigma\in E(2)$ and $4\varphi+\delta\sigma\in E(0)$. At the same time,
\[4\varphi+\delta\sigma=2(2\varphi+\delta\sigma)-\delta\sigma.\]
Since both $E(2)$ and $\im\delta$ are orthogonal to $E(0)$, it follows that $4\varphi+\delta\sigma=0$. Now
\[2\varphi+\delta\sigma=-2\varphi=\frac{1}{2}\delta\sigma\in E(2).\]
Furthermore $d\sigma=0$ and $\sigma\bot\ker\Delta\big|_{\Omega^3}$ imply that $\sigma\in\im d\big|_{\Omega^2}$. Since
\[(\sigma,d\eta)_{L^2}=(\delta\sigma,\eta)_{L^2}=(-4\varphi,\eta)_{L^2}=(-2\Delta\varphi,\eta)_{L^2}=(-2\delta d\varphi,\eta)_{L^2}=(-2d\varphi,d\eta)_{L^2}\]
for any $\eta\in\Omega^2$, it follows that $\sigma=-2d\varphi$. One can check that $(\varphi,-2d\varphi+\tau)$ solves (\ref{laplacesystem}) for any $\varphi\in E(2)$ and $\tau\in\ker\Delta\big|_{\Omega^3}$. Note that $\ker\Delta\big|_{\Omega^3}\subset\Omega^{2,1}_{0,\R}$ by a theorem of Verbitsky \cite[Thm.~6.2]{verbitsky}.

Let $\varepsilon=\frac{25}{4}$. In this case, $\mu_1=\mu_2$ and the matrix $A$ is not diagonalizable. If we set
\[\alpha':=\frac{1}{3}\delta\sigma,\quad\beta':=3\varphi+\delta\sigma,\]
then we obtain the system
\[\begin{pmatrix}
	\Delta\alpha'\\
   \Delta\beta'
  \end{pmatrix}=\begin{pmatrix}
   \frac{3}{4}&1\\
   0&\frac{3}{4}
  \end{pmatrix}\begin{pmatrix}
   \alpha'\\
   \beta'
  \end{pmatrix}.\]
For any $\beta\in E(\frac{3}{4})$, we therefore need to solve $(\Delta-\frac{3}{4})\alpha'=\beta'$. But $\im(\Delta-\frac{3}{4})\bot\ker(\Delta-\frac{3}{4})$, hence there can only exist a solution if $\beta'=0$. In this case, we obtain $\Delta\alpha'=\frac{3}{4}\alpha'$. Furthermore, $\ast d\sigma\in E(-\frac{1}{4})$. Since $\Delta$ is nonnegative, this implies that $d\sigma=0$. We can recover $\sigma$ from $\alpha'$ by
\[\frac{1}{4}\sigma=-\Delta\sigma-4d\varphi=-d\delta\sigma+\frac{4}{3}d\delta\sigma=d\alpha'\]
and $\varphi$ by $3\varphi+\delta\sigma=\beta'=0$. In total, the space of solutions is isomorphic to $E(\frac{3}{4})$.

Let $\varepsilon>\frac{25}{4}$. Then $\mu_1,\mu_2$ are imaginary and $\mu_3<0$. Since $\Delta$ is nonnegative, $E(\mu_i)=0$ for $i=1,2,3$. The mapping $\Psi$ is still an isomorphism, hence the space of solutions is trivial.
\end{proof}

\begin{bem}
 Note that in case (i) of the above lemma, the eigenvalues $\mu_i$ are subject to the bounds $\mu_1<12$, $\mu_2<2$ and $\mu_3<6$ if we assume that $\varepsilon>0$. In the critical case where we set $\varepsilon=0$, we recover the description
 \[\varepsilon(g)\cong E(2)\oplus E(6)\oplus E(12)\]
 by A. Moroianu and U. Semmelmann \cite[Thm.~5.1]{nearlykaehler}.
\end{bem}

Recall from Section \ref{sec:prelim3s} that a naturally reductive Riemannian $3$-symmetric space $(G/H,g)$ carries a nearly Kähler structure whose Hermitian connection $\bar\nabla$ coincides with the canonical reductive connection $\CR$ of the homogeneous structure. In fact, these assumptions hold for all of the homogeneous Gray manifolds, namely $S^6$, $S^3\times S^3$, $\CP^3$ and the flag manifold $F_{1,2}$.

The Hermitian Laplace operator $\bar\Delta$ therefore coincides with the standard Laplacian $\Lcr$. In light of Section \ref{sec:prelimha}, this enables us to describe the eigenspaces of $\bar\Delta$ in terms of irreducible complex representations of $G$.

However, the statement of Lemma~\ref{epsilon} involves the ei\-gen\-spa\-ces of $\Delta$ restricted to the subspace of \emph{coclosed} forms in $\Omega^{1,1}_{0,\R}$. We note that $\bar\Delta=\Delta$ on $\Omega^{1,1}_0\cap\ker\delta$ follows from combining Cor.~3.5 and Cor.~4.5 of \cite{nearlykaehler}. Thus it suffices to first search for small eigenvalues of $\bar\Delta$ on $\Omega^{1,1}_0$. To single out the coclosed elements in an eigenspace, we are going to perform explicit computations utilizing the following lemma. A similar formula for the divergence on symmetric tensors has already been employed to decide the stability of certain symmetric spaces, cf. \cite[Lem.~3.3]{schwahn} and \cite[Sec.~2]{sw}.

\begin{lem}\label{deltank6}
Let $(M=G/H,g,J)$ be a homogeneous Gray manifold with reductive decomposition $\g=\h\oplus\m$ as in Section \ref{sec:prelim3s}. Let $\delta: \Omega^{1,1}_0\to\Omega^1_\C$ denote the codifferential. Its prototypical differential operator is given by
\[\delta:\ \Hom_H(V_\gamma,\Lambda^{1,1}_0\m)\to\Hom_H(V_\gamma,\m^\C):\ F\mapsto\sum_ie_i\lrcorner F\circ(\rho_\gamma)_\ast(e_i)\]
for any orthonormal basis $(e_i)$ of $\m$.
\end{lem}
\begin{proof}
By (\ref{CR}), the Ambrose-Singer connection $\CR=\bar\nabla$ translates into a directional derivative on $C^\infty(G,\Lambda^{1,1}_0\m)$. Fix some dominant integral weight $\gamma\in\hat G$, vector $v\in V_\gamma$ and homomorphism $F\in\Hom_H(V_\gamma,\Lambda^{1,1}_0\m)$, and let $\alpha\in\Omega^{1,1}_0$ be associated to $v\otimes F$ via the equivalence (\ref{peterweyl}). Differentiating the smooth function
\[\hat\alpha:\ G\to\Lambda^{1,1}_0\m:\ x\mapsto F(\rho_\gamma^{-1}(x)v)\]
defined in (\ref{pwembedding}), we find that, for any $x\in G$ and $X\in\m\cong T_oM$,
\[\widehat{\bar\nabla_X\alpha}=X(\hat\alpha)=-F((\rho_\gamma)_\ast(X)v)\]
Let $(e_i)$ denote some local orthonormal frame of $TM$. It follows from \cite[Lem.~4.2]{nearlykaehler} that
\[\delta\alpha=-\sum_ie_i\lrcorner\nabla_{e_i}\alpha=-\sum_ie_i\lrcorner\bar\nabla_{e_i}\alpha\]
(essentially using that $\bar\nabla$ has skew torsion). Combining the above, we obtain
\begin{align*}
\widehat{\delta\alpha}&=-\sum_ie_i\lrcorner\widehat{\bar\nabla_{e_i}\alpha}=\sum_ie_i\lrcorner F((\rho_\gamma)_\ast(e_i)v).
\end{align*}
at the base point. The assertion now follows from the $G$-invariance of $\delta$.
\end{proof}

\begin{bem}\label{Ginv}
Recall that plugging harmonic $2$- and $3$-forms into the bundle isomorphisms (\ref{iso}) yields destabilizing directions for any Gray manifold \cite{nkstability}. Moreover, on a homogeneous Gray manifold $M=G/H$, harmonic $2$- and $3$-forms are always $G$-invariant. Indeed, $\ker\Delta\big|_{\Omega^2}\subset\Omega^{1,1}_{0,\R}$ and $\ker\Delta\big|_{\Omega^3}\subset\Omega^{2,1}_{0,\R}$ by \cite[Thm.~6.2]{verbitsky}. Since harmonic forms are coclosed, Cor.~3.5 and Cor.~4.5 of \cite{nearlykaehler} imply that these are also harmonic for the Hermitian Laplace operator $\bar\Delta$. In the homogeneous case, this means they lie in the kernel of $\Cas^G_\ell$, thus showing $G$-invariance. Since the isomorphisms (\ref{iso}) are $G$-equivariant, it follows that the destabilizing directions obtained from this construction are themselves $G$-invariant. As we will see in the following section, there are no other destabilizing directions, thus the coindex coincides with the $G$-invariant coindex in each case.
\end{bem}

\section{Case-by case stability analysis}
\label{sec:examples}

\subsection{Nearly Kähler $S^3\times S^3$}
\label{sec:s3xs3}

Let $K=\SU(2)$ with Lie algebra $\k=\su(2)$, let $G=K\times K\times K$ with Lie algebra $\g=\k\oplus\k\oplus\k$ and let $H=\Delta K\subset G$ be the diagonal, with Lie algebra $\h\cong\k$. We consider the homogeneous space $M=G/H$. Let $B_\k$ denote the Killing form of $\k$. The inner product on $\g$ that is given by $-\frac{1}{12}(B_\k\oplus B_\k\oplus B_\k)$ defines a normal Riemannian metric $g$ on $M$, which has scalar curvature $\scal_g=30$. The automorphism
\[\sigma: G\to G:\ (k_1,k_2,k_3)\mapsto(k_2,k_3,k_1)\]
that cyclically permutes the factors is of order three, fixes $H$ and hence gives $M$ the structure of a Riemannian $3$-symmetric space.

We denote by $E=\C^2$ the standard representation of $K=\SU(2)$. Furthermore, we label the irreducible complex representations of $K$ by $k\in\N_0$, where $V_k=\Sym^kE$ is the unique $(k+1)$-dimensional irreducible complex representation of $K$.

\begin{lem}\label{subcrit1}
 Let $V_{\gamma}$ be an irreducible complex representation of $G$ with $\Cas^G_\gamma<12$ and
 \[\Hom_H(V_\gamma,\Lambda^{1,1}_0\m)\neq0.\]
 Then $V_\gamma$ is equivalent to one of the representations $E\otimes E\otimes\C$, $E\otimes\C\otimes E$ and $\C\otimes E\otimes E$ of $G$. In any of those cases,
 \[\dim\Hom_H(V_\gamma,\Lambda^{1,1}_0\m)=1\]
 and the Casimir eigenvalue is $\Cas_\gamma^G=9$.
\end{lem}
\begin{proof}
Since irreducible representations of $G=K\times K\times K$ are precisely the threefold tensor products of irreducible representations of $K$, we can label them by
\[V_{(a,b,c)}:=V_a\otimes V_b\otimes V_c,\]
where $a,b,c\in\N_0$. Restricting the representation $V_{(a,b,c)}$ to the diagonal $H\subset G$ simply yields the tensor product $V_a\otimes V_b\otimes V_c$ as a representation of $K$. The Clebsch-Gordan rules allow us to decompose these into irreducible summands. From \cite[(31)]{hermitianlaplace}, we know that the Casimir eigenvalues of $G$ with respect to the inner product $-\frac{1}{12}(B_\k\oplus B_\k\oplus B_\k)$ are given by
\[\Cas^G_{(a,b,c)}=\Cas^{\SU(2)}_a+\Cas^{\SU(2)}_b+\Cas^{\SU(2)}_c=\frac{3}{2}(a(a+2)+b(b+2)+c(c+2))\]
for $a,b,c\in\N_0$. The results for the first few Casimir eigenvalues are listed in Table~\ref{CasG}.

\begin{table}[h]
\centering
\begin{tabular}{c|c|c}
$\gamma=(a,b,c)$&Branching of $V_\gamma$ to $K$&$\Cas_\gamma^G$\\\hline
$(0,0,0)$&$\C=V_0$&0\\
$(1,0,0),(0,1,0),(0,0,1)$&$V_1\otimes\C\otimes\C=V_1$&$\frac{9}{2}$\\
$(1,1,0),(1,0,1),(0,1,1)$&$V_1\otimes V_1\otimes\C\cong V_2\oplus V_0$&$9$\\
$(1,1,1)$&$V_1\otimes V_1\otimes V_1\cong V_3\oplus V_1\oplus V_1$&$\frac{27}{2}$\\
$(2,0,0),(0,2,0),(0,0,2)$&$V_2\otimes\C\otimes\C=V_2$&$12$
\end{tabular}
\caption{The first few Casimir eigenvalues of $G=K\times K\times K$}
\label{CasG}
\end{table}

By \cite[Lem. 5.5]{hermitianlaplace}, we know that $\Lambda^{1,1}_0\m\cong V_4\oplus V_2$ as a representation of $K$. Comparing summands now yields that the only irreducible complex representations $V_\gamma$ of $G$ with Casimir eigenvalue smaller than $12$ and nontrivial $\Hom_H(V_\gamma,\Lambda^{1,1}_0\m)$ are $V_{(1,1,0)}$, $V_{(1,0,1)}$ and $V_{(0,1,1)}$.
\end{proof}

Since $\bar\Delta$ acts as the Casimir operator, this means that $9$ is the only eigenvalue smaller than $12$ in the spectrum of $\bar\Delta$ on $\Omega^{1,1}_0$. The eigenvalue $12$ itself does also occur on $\Omega^{1,1}_0$, but in \cite{hermitianlaplace} it is shown that the corresponding eigenforms are not coclosed, hence proving that $(M,g)$ has no infinitesimal Einstein deformations. It now remains to check whether this is the case for the eigenforms to the eigenvalue $9$.

\begin{lem}\label{9coclosed}
 The eigenspace of $\bar\Delta$ on $\Omega^{1,1}_{0,\R}$ to the eigenvalue $9$ contains no nontrivial coclosed forms.
\end{lem}
\begin{proof}
We explicitly calculate the codifferential on the summands in question using the formula from Lemma~\ref{deltank6}.

Lemma~\ref{subcrit1} tells us that the relevant summands of the left-regular representation on $\Omega^{1,1}_0$ are $V_{(1,1,0)}$, $V_{(1,0,1)}$ and $V_{(0,1,1)}$. First, the representation $V_{(1,1,0)}=E\otimes E\otimes\C$ is given by
\[\rho: G\to\Aut(E\otimes E):\ \rho(k_1,k_2,k_3)(v_1\otimes v_2)=k_1v_1\otimes k_2v_2\]
for any $k_1,k_2,k_3\in K$ and $v_1,v_2\in E$. Recall that $\m^\C=\m^+\oplus\m^-$, where $\m^+$ is the eigenspace of $\sigma_\ast$ to the eigenvalue $\j=-\frac{1}{2}+\frac{3}{2}\i$, and $\m^-$ is the eigenspace to $\j^2=-\frac{1}{2}-\frac{3}{2}\i$. Explicitly,
\[\m^+=\{(Y,\j Y,\j^2Y)\,|\,Y\in\k\},\quad\m^-=\{(Y,\j^2Y,\j Y)\,|\,Y\in\k\}.\]
Let $(Y_1,Y_2,Y_3)$ be an orthonormal basis of $\k$ with respect to the inner product $-B_{\k}$. Then
\[X_i:=2(Y_i,\j Y_i,\j^2Y_i)\in\m^+,\quad\overline{X_i}:=2(Y_i,\j^2Y_i,\j Y_i)\in\m^-,\quad i=1,2,3\]
constitute an orthonormal basis of $\m^\C$ with respect to $-\frac{1}{12}(B_\k\oplus B_\k\oplus B_\k)$. With respect to the basis
\[\mathcal{B}=(z_1\otimes z_1,z_1\otimes z_2,z_2\otimes z_1,z_2\otimes z_2)\quad\text{of}\quad E\otimes E,\]
we can represent $\rho_\ast(X_i)$ by the $4\times4$-matrices
\begin{align*}
\rho_\ast(X_1)&=\frac{\i}{\sqrt{2}}\begin{pmatrix}
                               0&\j&1&0\\
                               \j&0&0&1\\
                               1&0&0&\j\\
                               0&1&\j&0
                              \end{pmatrix},&\rho_\ast(X_2)&=\frac{1}{\sqrt{2}}\begin{pmatrix}
                               0&-\j&-1&0\\
                               \j&0&0&-1\\
                               1&0&0&-\j\\
                               0&1&\j&0
                              \end{pmatrix},
\end{align*}
\begin{equation}
\rho_\ast(X_3)=\frac{\i}{\sqrt{2}}\begin{pmatrix}
                               1+\j&0&0&0\\
                               0&1-\j&0&0\\
                               0&0&-1+\j&0\\
                               0&0&0&-1-\j
                              \end{pmatrix}.\label{s3s3action}
\end{equation}
This works similarly for $\overline{X_i}$ by means of simply replacing the symbol $\j$ with $\j^2$. 

Turning to the decomposition of $\Lambda^{1,1}\m$ and $E\otimes E$ into $K$-irreducible summands, we have
\begin{align*}
\Lambda^{1,1}\m&=\m^+\otimes\m^-\cong\k^\C\otimes\k^\C\cong\Sym^2_0\k^\C\oplus\Lambda^2\k^\C\oplus\C,\\
E\otimes E&=\Sym^2 E\oplus\Lambda^2 E 
\end{align*}
with common summand $\Lambda^2\k^\C\cong\k^\C\cong V_2=\Sym^2E$. If we choose the basis of the image of $\Lambda^2\k^\C$ in $\Lambda^{1,1}\m$ as
\[\mathcal{B}'=\left(X_1\wedge\overline{X_2}-X_2\wedge\overline{X_1},\ X_2\wedge\overline{X_3}-X_3\wedge\overline{X_2},\ X_3\wedge\overline{X_1}-X_1\wedge\overline{X_3}\right),\]
then a generator of $\Hom_K(E\otimes E,\Lambda^{1,1}_0\m)$ is represented by the matrix
\[F=\frac{1}{\sqrt{2}}\begin{pmatrix}
   0&1&1&0\\
   1&0&0&1\\
   -\i&0&0&\i
  \end{pmatrix}
\]
with respect to $\mathcal{B}$ and $\mathcal{B}'$. Taking
\[\mathcal{B}'':=\left(X_1,X_2,X_3,\overline{X_1},\overline{X_2},\overline{X_3}\right)\]
as a basis of $\m^\C$, we compute $\delta(F)$ according to Lemma~\ref{deltank6}:
\begin{align*}
\delta(F)&=\sum_iX_i\lrcorner(F\circ \rho_\ast(X_i))+\sum_i\overline{X_i}\lrcorner(F\circ \rho_\ast(\overline{X_i}))=\begin{pmatrix}
           0&0&0&0\\
           0&0&0&0\\
           0&1-\j^2&-1+\j^2&0\\
           0&0&0&0\\
           0&0&0&0\\
           0&1-\j&-1+\j&0
          \end{pmatrix}
\end{align*}
with respect to the bases $\mathcal{B}$ and $\mathcal{B}''$. We have thus shown that
\[\delta: \Hom_H(V_{(1,1,0)},\Lambda^{1,1}_0\m)\to\Hom_H(V_{(1,1,0)},\m^\C)\]
does not vanish. For the other two representations $V_{(1,0,1)}$ and $V_{(0,1,1)}$ modeled on the vector space $E\otimes E$ with the same Casimir eigenvalue, the only thing that changes is the action of $G$, amounting to cyclic permutations of the factors $1,\j,\j^2$ in (\ref{s3s3action}). The computations work out analogously. We conclude that the eigenspace of $\bar\Delta$ on $\Omega^{1,1}_0$ to the eigenvalue $9$ contains no nonzero coclosed forms.
\end{proof}

\noindent
It remains to apply Lemma~\ref{epsilon} in order to finally obtain the desired result.

\begin{prop}\label{s3xs3prop}
  On the nearly Kähler manifold $S^3\times S^3$, the space of destabilizing directions for its Einstein metric $g$ consists solely of the $2$-dimensional $\Delta_L$-eigenspace to the eigenvalue $4$, the latter arising from harmonic $3$-forms. In total, the coindex of $g$ is $2$.
\end{prop}
\begin{proof}
We recall the bounds $\mu_1<12$, $\mu_2<2$ and $\mu_3<6$ from Lemma~\ref{epsilon}. Lemmas \ref{subcrit1} and \ref{9coclosed} imply that $E(\mu)$ is trivial for all $\mu<12$. However, \ref{epsilon}, (ii) yields a space of solutions isomorphic to the space of harmonic $3$-forms in the case $\varepsilon=6$. Since $b_3=2$, this gives us a $2$-dimensional subspace of $\TT$ such that
\[\Delta_Lh=(10-6)h=4h\]
for any element $h$.
\end{proof}

\begin{bem}\label{s3xs3Gstr}
As seen in Remark~\ref{Ginv}, the destabilizing directions, which come from harmonic $3$-forms, are $G$-invariant. Conversely, all $G$-invariant traceless variations of the metric are destabilizing. Since $\m\cong\su(2)\oplus\su(2)$, Schur's Lemma implies that
\[\S^2_0(M)^G\cong(\Sym^2_0\m)^H\cong\Sym^2_0\R^2\cong\R^2.\]
Thus $\S^2_0(M)^G$ is already exhausted by the two-dimensional space of destabilizing directions, meaning that $S^3\times S^3$ is $G$-strongly unstable in the sense of \cite{ejlauret}.
\end{bem}

\subsection{Nearly Kähler $\CP^3$}
\label{sec:cp3}

Let $G=\SO(5)$ and $H=\U(2)$. We consider $H$ embedded into $G$ via the natural inclusions $\U(2)\subset\SO(4)\subset\SO(5)$. The normal Riemannian metric $g$ induced by $-\frac{1}{12}B_{\g}$ on the homogeneous space $M=G/H$ is the nearly Kähler metric on $\CP^3$, normalized to $\scal_g=30$. It should be noted that $(M,g)$ is, again, naturally reductive and 3-symmetric, with reductive complement $\m=\h^\perp$.

Let $\t=\{\diag(\i\theta_1,\i\theta_2)\,|\,\theta_1,\theta_2\in\R\}\subset\h\subset\g$ be the maximal torus Lie algebra. The positive roots $\alpha_i\in\t^\ast$ of $G$ can then be expressed as
\[\alpha_1=\theta_1,\ \alpha_2=\theta_2,\ \alpha_3=\theta_1+\theta_2,\ \alpha_4=\theta_1-\theta_2.\]
In terms of root spaces of $G$, we have
\begin{equation}
\m^\C=\g^{\alpha_1}\oplus\g^{-\alpha_1}\oplus\g^{\alpha_2}\oplus\g^{-\alpha_2}\oplus\g^{\alpha_3}\oplus\g^{-\alpha_3}.\label{cp3root}
\end{equation}
The almost complex structure $J$ can be defined by specifying its $\pm\i$-eigenspaces
\[\m^+=\g^{\alpha_1}\oplus\g^{\alpha_2}\oplus\g^{-\alpha_3},\ \m^-=\g^{-\alpha_1}\oplus\g^{-\alpha_2}\oplus\g^{\alpha_3}.\]
In passing, we note that the standard complex structure on $\CP^3$ has $\m^+=\g^{\alpha_1}\oplus\g^{\alpha_2}\oplus\g^{\alpha_3}$.

We label the irreducible complex representations $V_\gamma$ of $G$ by their highest weights $\gamma=(a,b)$, where $a,b\in\N_0$, $a\geq b$. For example, $V_{(1,0)}=\C^5$ is the complexified standard representation of $G$, while $V_{(1,1)}=\so(5)^\C$ is the complexified adjoint representation.

Again, let $E=\C^2$ be the standard representation of $\SU(2)$. Let furthermore $\C_k$ denote the representation of $\U(1)$ on $\C$ defined by
\[\U(1)\times\C\to\C:\ (z,w)\mapsto z^kw\]
for any $k\in\Z$. The irreducible complex representations of $H\cong(\SU(2)\times\U(1))/\Z_2$ are then given by $E^a_b:=\Sym^aE\otimes\C_b$ for $a\in\N_0$ and $b\in\Z$, $a\equiv b\mod 2$.

\begin{lem}\label{subcrit2}
 Let $V_{\gamma}$ be an irreducible complex representation of $G$ with $\Cas^G_\gamma<12$ and
 \[\Hom_H(V_\gamma,\Lambda^{1,1}_0\m)\neq0.\]
 Then $V_\gamma$ is equivalent to either the trivial representation $V_{(0,0)}$ or the standard representation $V_{(1,0)}$. In both cases,
 \[\dim\Hom_H(V_\gamma,\Lambda^{1,1}_0\m)=1\]
 and the Casimir eigenvalues are $\Cas_{(0,0)}^G=0$ and $\Cas_{(1,0)}^G=8$, respectively.
\end{lem}
\begin{proof}
We first work out how to decompose the restriction of $V_{(1,0)}$ to $H$ into irreducible summands. We know that $V_{(1,0)}=\C^5$ is the complexified standard representation of $G$. The inclusion $\U(2)\subset\SO(4)$ can be understood as realification $(E^1_1)^\R$ of the defining representation $E^1_1$ of $\U(2)$. Furthermore, the inclusion $\SO(4)\subset\SO(5)$ defines a five-dimensional real representation $\R^4\oplus\R$ of $\SO(4)$, where the group acts as on its defining representation on the first summand and trivially on the second. In total, the restriction of the real standard representation of $G=\SO(5)$ to $H=\U(2)$ is given by $(E^1_1)^\R\oplus\R$. Complexifying then yields the decomposition
\[V_{(1,0)}=(E^1_1)^{\R\C}\oplus\C\cong E^1_1\oplus E^1_{-1}\oplus E^0_0.\]
For the branching of $V_{(1,1)}$ under the restriction to $H$, we refer to \cite[Lem.~5.9]{hermitianlaplace}. Using the decomposition
\[V_{(1,0)}\otimes V_{(1,0)}\cong V_{(2,0)}\oplus V_{(1,1)}\oplus V_{(0,0)}\]
of $G$-representations, the known branchings of $V_{(1,1)}$ and $V_{(1,0)}$ as well as the Clebsch-Gordan rules, we can also work out the branching of $V_{(2,0)}$ to $H$, although it will not be needed hereafter.

By \cite[(33)]{hermitianlaplace}, we know that on the irreducible representation $V_{(a,b)}$ of $G$, the Casimir eigenvalue of $G$ with respect to the inner product $-\frac{1}{12}B_{\g}$ is given by
\[\Cas^G_{(a,b)}=2(a(a+3)+b(b+1)).\]
Table \ref{CasSO5} lists the results for the smallest few Casimir eigenvalues of $G$.

\begin{table}[h]
\centering
\begin{tabular}{c|c|c}
$\gamma=(a,b)$&Branching of $V_\gamma$ to $H$&$\Cas^G_\gamma$\\\hline
$(0,0)$&$\C=E^0_0$&0\\
$(1,0)$&$E^1_1\oplus E^1_{-1}\oplus E^0_0$&$8$\\
$(1,1)$&$E^2_0\oplus E^1_1\oplus E^1_{-1}\oplus E^0_2\oplus E^0_0\oplus E^0_{-2}$&$12$\\
$(2,0)$&$E^2_2\oplus E^2_0\oplus E^2_{-2}\oplus E^1_1\oplus E^1_{-1}\oplus E^0_0$&$20$\\
\end{tabular}
\caption{The first few Casimir eigenvalues of $G=\SO(5)$}
\label{CasSO5}
\end{table}

\cite[Lem. 5.8]{hermitianlaplace} tells us that
\[\Lambda^{1,1}_0\m\cong E^2_0\oplus E^1_3\oplus E^1_{−3}\oplus E^0_0\]
as a representation of $H=\U(2)$. By comparing summands, we conclude that $V_{(0,0)}$ and $V_{(1,0)}$ are the only irreducible complex representations $V_\gamma$ of $G$ with Casimir eigenvalue smaller than $12$ and nontrivial $\Hom_H(V_\gamma,\Lambda^{1,1}_0\m)$.
\end{proof}

Again, $\bar\Delta$ acts as the Casimir operator. We therefore know that $0$ and $8$ are the only eigenvalues smaller than $12$ in the spectrum of $\bar\Delta$ on $\Omega^{1,1}_0$. As in the case $S^3\times S^3$, the eigenvalue $12$ does occur on $\Omega^{1,1}_0$, but the corresponding eigenforms are not coclosed (see \cite{hermitianlaplace}). Hence, $(M,g)$ has no infinitesimal Einstein deformations. It remains to check whether the eigenforms to the eigenvalues $0$ and $8$ are coclosed.

\begin{lem}\label{cp3coclosed}
 The eigenspace of $\bar\Delta$ on $\Omega^{1,1}_{0,\R}$ to the eigenvalue $0$ consists of coclosed forms, while the eigenspace to the eigenvalue $8$ contains no nontrivial coclosed forms.
\end{lem}
\begin{proof}
The eigenspace of $\Delta$ to the eigenvalue $0$ corresponds to the trivial summand in the left-regular representation, i.e. to $G$-invariant elements of $\Omega^{1,1}_0$. But these are parallel with respect to the Ambrose-Singer connection (which equals the canonical Hermitian connection $\bar\nabla$). Recall that
\[\delta=-\sum_ie_i\lrcorner\nabla_{e_i}=-\sum_ie_i\lrcorner\bar\nabla_{e_i}\qquad\text{on }\Omega^{1,1}_0\]
 by \cite[Lem.~4.2]{nearlykaehler}. It follows any element in the $0$-eigenspace of $\bar\Delta$ on $\Omega^{1,1}_0$ is coclosed.

For the eigenspace to the eigenvalue $8$, we again make use of the formula from Lemma~\ref{deltank6} in an explicit calculation. Lemma~\ref{subcrit2} tells us that the relevant summand of the left-regular representation on $\Omega^{1,1}_0$ is $V_{(1,0)}$. A generator $F$ of the one-dimensional space $\Hom_H(V_{(1,0)},\Lambda^{1,1}_0\m)$ must map the trivial summand $E^0_0$ in the $H$-representation
\[V_{(1,0)}\cong E^1_1\oplus E^1_{-1}\oplus E^0_0\]
to the trivial summand in
\[\Lambda^{1,1}_0\m\cong E^2_0\oplus E^1_3\oplus E^1_{−3}\oplus E^0_0.\]
The former is spanned by $v_5$, where $(v_1,\ldots,v_5)$ is the standard basis of $V_{(1,0)}=\C^5$. To describe the latter, we remark that the root spaces in decomposition (\ref{cp3root}) can be written as
\begin{align*}
\g_{\alpha_1}&=\spann\{e_1-\i e_2\},&\g_{-\alpha_1}&=\spann\{e_1+\i e_2\},\\
\g_{\alpha_2}&=\spann\{e_3-\i e_4\},&\g_{-\alpha_2}&=\spann\{e_3+\i e_4\},\\
\g_{\alpha_3}&=\spann\{f_1-\i f_2\},&\g_{-\alpha_3}&=\spann\{f_1+\i f_2\}
\end{align*}
in terms of the basis
\begin{align*}
e_1&:=E_{15}-E_{51},&e_2&:=E_{25}-E_{52},&e_3&:=E_{35}-E_{53},\\
e_4&:=E_{45}-E_{54},&f_1&:=E_{13}-E_{24}-E_{31}+E_{42},&f_2&:=E_{14}+E_{23}-E_{32}-E_{41}
\end{align*}
of $\m\subset\so(5)$, where $E_{ij}$ denotes the $5\times5$-matrix with $1$ at position $(i,j)$ and zero at all other entries. Note that under the inner product $-\frac{1}{12}B_{\g}$, the basis $(\sqrt{2}e_i,f_j)$ is orthonormal. It follows from the definition of $J$ in terms of $\m^\pm$ that the Kähler form can be written as
\[\omega=2e_{12}+2e_{34}-f_{12}.\]
By \cite[(32)]{hermitianlaplace}, the root space $\g^{-\alpha_3}\subset\g^\C$ is $H$-invariant and equivalent to $E^0_{-2}$. By conjugation, $\g^{\alpha_3}\cong E^0_2$. Hence $H$ acts trivially on $\g^{-\alpha_3}\otimes\g^{\alpha_3}\cong E^0_{-2}\otimes E^0_2=E^0_0$. Recall that $\g^{-\alpha_3}\subset\m^+$ and $\g^{\alpha_3}\subset\m^-$. It follows that $f_{12}=\frac{\i}{2}(f_1+\i f_2)\wedge(f_1-\i f_2)$ spans a trivial subspace of $\Lambda^{1,1}\m=\m^+\wedge\m^-$. Since $\Lambda^{1,1}_0\m$ is the orthogonal complement of $\omega$ in $\Lambda^{1,1}\m$, the remaining trivial summand in $\Lambda^{1,1}_0\m$ must be spanned by
\[\eta=e_{12}+e_{34}+f_{12}.\]
Having found the $H$-trivial subspaces of $V_{(1,0)}$ and $\Lambda^{1,1}_0\m$, we see that $\Hom_H(V_{(1,0)},\Lambda^{1,1}_0\m)$ is spanned by
\[F:\ \C^5\to\Lambda^{1,1}_0\m:\ (z_1,\ldots,z_5)\mapsto z_5\eta.\]
Observe that $f_jv_k\perp v_5$ for all $j,k$ and $e_iv_k\perp v_5$ for all $i,k$, except for
\begin{align*}
e_1v_1=e_2v_2=e_3v_3=e_4v_4=-v_5.
\end{align*}
Hence, by Lemma~\ref{deltank6},
\[\langle\delta(F)(v_i),X\rangle=2\langle F(e_iv_i),e_i\wedge X\rangle=-2\langle\eta,e_i\wedge X\rangle=-2\langle e_i\lrcorner\eta,X\rangle\]
for any $X\in\m^\C$. We have thus shown that
\[\delta: \Hom_H(V_{1,0},\Lambda^{1,1}_0\m)\to\Hom_H(V_{1,0},\m^\C)\]
is not the zero map -- in fact, it maps $\delta(F)=F'$, where
\[F'(v_i)=-2e_i\lrcorner\eta\text{ for }i=1,\ldots,4,\quad F'(v_5)=0.\]
Thus, the eigenspace of $\bar\Delta$ on $\Omega^{1,1}_0$ to the eigenvalue $8$ contains no nonzero coclosed forms.
\end{proof}

\noindent
As before, the desired result follows from an application of Lemma~\ref{epsilon}.

\begin{prop}\label{cp3prop}
  On the nearly Kähler manifold $\CP^3$, the space of destabilizing directions for its Einstein metric $g$ consists solely of the $1$-dimensional $\Delta_L$-eigenspace to the eigenvalue $6$, arising from harmonic $2$-forms. Consequently, the coindex of $g$ is $1$.
\end{prop}
\begin{proof}
Lemmas \ref{subcrit2} and \ref{cp3coclosed} imply that $E(\mu)$ from Lemma~\ref{epsilon} is trivial for all $\mu<12$ except if $\mu=0$. We have already seen that $E(0)$ consists of harmonic forms. In fact, \cite[Thm.~6.2]{verbitsky} implies that all harmonic $2$-forms on $M$ lie in $E(0)$. If we solve
\[\mu_{1,2}=7-\varepsilon\pm\sqrt{25-\varepsilon}=0\]
for $\varepsilon$, we obtain $\varepsilon=5\pm 1$; from $\mu_3=6-\varepsilon=0$ we obtain $\varepsilon=6$. If $\varepsilon=6$, i.e. in case (ii) of Lemma~\ref{epsilon}, the space of solutions to (\ref{laplacesystem}) is isomorphic to $E(2)\oplus\ker\Delta\big|_{\Omega^3}$. Since $2$ does not appear in the spectrum on $\Omega^{1,1}_0$ and $b_3(\CP^3)=0$, this space is trivial. The only possibility in which $E(0)$ contributes to the space of solutions of (\ref{laplacesystem}) is in case (i) of Lemma~\ref{epsilon} with $\varepsilon=4$. Since $b_2=1$, we obtain a $1$-dimensional subspace of $\TT$ on which
\[\Delta_Lh=(10-4)h=6h\]
for any of its elements $h$.
\end{proof}

\subsection{The flag manifold $F_{1,2}$}
\label{sec:flag}

Let $G=\SU(3)$ and $H=T^2$, the latter embedded into $G$ via
\[H=\U(1)\times\U(1)\hooklongrightarrow G:\ (z_1,z_2)\mapsto\diag(z_1,z_2,(z_1z_2)^{-1}).\]
The homogeneous space $M=G/H$ is a description of the manifold $F_{1,2}$ of flags in $\C^3$. As in the previous examples, we endow $M$ with the normal Riemannian metric $g$ induced by $-\frac{1}{12}B_\g$, which has scalar curvature $\scal_g=30$. The Riemannian homogeneous space $(M,g)$ is naturally reductive, $3$-symmetric and hence nearly Kähler. For more information on the nearly Kähler structure, see \cite[Sec.~5.6]{hermitianlaplace}.

Denote by $E=\C^3$ the standard representation of $G$. Any irreducible complex representation of $G$ can then be described as the Cartan summand $V_{(k,l)}$ of the tensor product $\Sym^kE\otimes\Sym^l\bar E$ for some $k,l\in\N_0$. For example, $V_{(1,1)}$ is equivalent to the complexified adjoint representation $\su(3)^\C$ of $G$.

\begin{lem}\label{subcrit3}
 Let $V_{\gamma}$ be an irreducible complex representation of $G$ with $\Cas^G_\gamma<12$ and
 \[\Hom_H(V_\gamma,\Lambda^{1,1}_0\m)\neq0.\]
 Then $V_\gamma$ is the trivial representation and $\dim\Hom_H(\C,\Lambda^{1,1}_0\m)=2$.
\end{lem}
\begin{proof}
It follows from \cite[(35)]{hermitianlaplace} that the Casimir eigenvalue on the irreducible representation $V_{(k,l)}$ of $G$ with respect to the inner product $-\frac{1}{12}B_\g$ is given by
\[\Cas_{(k,l)}^G=2(k(k+2)+l(l+2)).\]

\begin{table}[h]
\centering
\begin{tabular}{c|c|c}
$(k,l)$&$\dim\Hom_H(V_{(k,l)},\Lambda^{1,1}_0\m)$&$\Cas_{(k,l)}^G$\\\hline
$(0,0)$&$2$&$0$\\
$(1,0)$&$0$&$6$\\
$(0,1)$&$0$&$6$\\
$(1,1)$&$4$&$12$\\
\end{tabular}
\caption{The first few Casimir eigenvalues of $G=\SU(3)$}
\label{CasSU3}
\end{table}

By analyzing the weights of the respective representations, it has already been checked in \cite[Sec.~5.6]{hermitianlaplace} that 
\[\Hom_H(V_{(1,0)},\Lambda^{1,1}_0\m)=\Hom_H(V_{(0,1)},\Lambda^{1,1}_0\m)=0\]
and $\dim\Hom_H(V_{(1,1)},\Lambda^{1,1}_0\m)=4$.

The maximal subspace of $\Omega^{1,1}_0$ on which $G$ acts trivially is precisely the kernel of the Casimir operator. We can argue exactly as in the proof of Lemma~\ref{cp3coclosed} that the elements of this space are coclosed. Furthermore, we recall that $\Delta=\bar\Delta$ on coclosed primitive $(1,1)$-forms. Since $\bar\Delta$ acts as the Casimir operator, we are simply talking about the subspace of harmonic forms in $\Omega^{1,1}_0$. By Verbitsky's theorem \cite[Thm.~6.2]{verbitsky}, all harmonic $2$-forms lie in $\Omega^{1,1}_0$. Thus, the multiplicity of the trivial representation $\C$ in $\Omega^{1,1}_0$ is $b_2(F_{1,2})=2$. By Frobenius reciprocity, this is also the dimension of $\Hom_H(\C,\Lambda^{1,1}_0\m)$.

The results of the above discussion are listed in Table \ref{CasSU3}.
\end{proof}

We have thus shown that $0$ is the only eigenvalue smaller than $12$ in the spectrum of $\bar\Delta$ on $\Omega^{1,1}_0$. Besides, it has been proven in \cite[Sec.~6]{hermitianlaplace} that an $8$-dimensional subspace of the $32$-dimensional eigenspace to the eigenvalue $12$ consists of coclosed forms and hence yields infinitesimal Einstein deformations of $(M,g)$. We will describe these more explicitly in Section \ref{sec:rigidity}.

\begin{prop}
  On the nearly Kähler manifold $F_{1,2}$, the space of destabilizing directions for its Einstein metric $g$ consists solely of a $2$-dimensional $\Delta_L$-eigenspace to the eigenvalue $6$, arising from harmonic $2$-forms. In total, the coindex of $g$ is $2$.
\end{prop}
\begin{proof}
One last time, we want to apply Lemma~\ref{epsilon}. By Lemma~\ref{subcrit3} and the fact that harmonic forms are coclosed, we see that for $\mu<12$, the eigenspace $E(\mu)$ is only nontrivial if $\mu=0$. With the same reasoning as in the proof of Prop.~\ref{cp3prop}, we conclude that (\ref{laplacesystem}) can be solved for $\varepsilon=4$, yielding a $2$-dimensional subspace of $\TT$ such that
\[\Delta_Lh=(10-4)h=6h\]
for all its elements $h$.
\end{proof}

\begin{bem}\label{flagcs}
The space of destabilizing directions (or of harmonic $2$-forms) can be described rather explicitly. The $3$-dimensional space
\[\Lambda^{1,1}\m^H\cong\Hom_H(V_{(0,0)},\Lambda^{1,1}\m)\]
of $H$-invariant elements of $\Lambda^{1,1}\m$ corresponds to the space $(\Omega^{1,1})^G$ of $G$-invariant $(1,1)$-forms on $F_{1,2}$ and is spanned by
\[e_{12},\quad e_{34},\quad e_{56},\]
using the notation introduced in Section \ref{sec:rigidity}. The Kähler form $\omega$ corresponding to the strict nearly Kähler structure on $F_{1,2}$ is given by
\[\hat\omega=e_{12}-e_{34}+e_{56}.\]
The $2$-dimensional space $\Lambda^{1,1}_0\m^H\cong(\Omega^{1,1}_0)^G=\ker\Delta\big|_{\Omega^2}$ responsible for the destabilizing directions is now the orthogonal complement of $\hat\omega$ in $\Lambda^{1,1}\m^H$.

The twistor space over $\CP^2$ can be identified with $F_{1,2}$ in three distinct ways, giving rise to three fibrations $F_{1,2}\to\CP^2$, which in turn induce six almost complex structures on $F_{1,2}$. Three of them are actually integrable with respective Kähler forms $\omega_1,\omega_2,\omega_3$, described by
\[\hat\omega_1=-e_{12}-e_{34}+e_{56},\quad\hat\omega_2=e_{12}+e_{34}+e_{56},\quad\hat\omega_3=e_{12}-e_{34}-e_{56},\]
while the other three coincide with the almost complex structure with Kähler form $\omega$. See \cite[Sec.~3.2.3]{morris} for a detailed description.

Thus, in light of the construction in Section \ref{sec:general} using the isomorphism
\[\S^+_0\to\Omega^{1,1}_{0,\R}:\ h\mapsto h\circ J,\]
the destabilizing directions of the nearly Kähler metric $g$ on $F_{1,2}$ can be viewed as coming from variations of $\omega=g(J\cdot,\cdot)$ in the directions of $\omega_1,\omega_2,\omega_3$ while fixing $J$.

Alternatively, consider the canonical variation, i.e. change of scale of fiber against base, on each of the three aforementioned fibrations
\[\pi_j:\ F_{1,2}=\frac{\SU(3)}{T^2}\longrightarrow\frac{\SU(3)}{\mathrm{S}(\U(2)\times\U(1))}=\CP^2,\qquad j=1,2,3.\]
In \cite[Prop.~4.4]{WW}, it is shown that these variations yield destabilizing tt-tensors. General destabilizing directions $h\in\TT$ are thus (at the base point) of the form
\[h=t_1g\big|_{\m_1}+t_2g\big|_{\m_2}+t_3g\big|_{\m_3},\qquad\m=\m_1\oplus\m_2\oplus\m_3,\]
with $t_1+t_2+t_3=0$, where each of the pairwise orthogonal subspaces $\m_j$ is the vertical tangent space with respect to $\pi_j$. Since $\pi_j$ are Riemannian submersions with totally geodesic fibers, the destabilizing directions are \emph{Killing tensors} by \cite[Ex.~7.3]{killing}, that is, they satisfy the Killing equation\footnote{This can be verified directly, using the description of $\nabla_X-\bar\nabla_X$ in Section \ref{sec:flagobstr} and the fact that $h$ is $G$-invariant and hence $\bar\nabla$-parallel.}
\[\nabla_Xh(X,X)=0\qquad\forall X\in TM.\]
\end{bem}

\section{Rigidity of $F_{1,2}$}
\label{sec:rigidity}

\subsection{The infinitesimal Einstein deformations of $F_{1,2}$}
\label{sec:flagied}

We will utilize the explicit description of the infinitesimal Einstein deformations of $F_{1,2}$ given in \cite[Sec.~6]{hermitianlaplace}. For this, it is helpful to represent $M=F_{1,2}$ as a quotient of $G=\U(3)$ by the diagonally embedded torus $H=T^3$.

Denote by $E_{ij}$ the $3\times 3$-matrix with a $1$ at position $(i,j)$ and zero entries elsewhere. Let $\{h_1,h_2,h_3,e_1,\ldots,e_6\}$ be the basis of $\g=\u(3)$ given by
\begin{align*}
h_1&=\i E_{11},&h_2&=\i E_{22},&h_3&=\i E_{33},\\
e_1&=E_{12}-E_{21},&e_2&=\i(E_{12}+E_{21}),&e_3&=E_{13}-E_{31},\\
e_4&=\i(E_{13}+E_{31}),&e_5&=E_{23}-E_{32},&e_6&=\i(E_{23}+E_{32}).
\end{align*}
Note that $\{h_1,h_2,h_3\}$ span the Lie algebra $\h\subset\g$, while the reductive complement $\m\subset\g$ is spanned by $\{e_1,\ldots,e_6\}$. We now define the inner product $\langle\cdot,\cdot\rangle$ on $\g$ (and the induced bi-invariant metric on $G$) in such a way that $(e_i,\sqrt{2}h_j)$ is an orthonormal system. One easily checks that this coincides with $-\frac{1}{12}B_{\su(3)}$ when restricted to $\su(3)\subset\g$, so we recover the same metric $g$ on $F_{1,2}$.

The space $\varepsilon(g)$ of infinitesimal Einstein deformations of $g$ is equivalent to $\su(3)$ via the following prodecure. For a fixed element $\xi\in\su(3)\subset\g$, let $\xi^\ast\in C^\infty(G,\g)$ be given by $\xi^\ast(x)=\Ad(x)\xi$. This defines smooth, real-valued functions $x_1,x_2,x_3,v_1,\ldots,v_6$ on $G$ via
\[\xi^\ast=\begin{pmatrix}
            2\i v_1&x_1+\i x_2&x_3+\i x_4\\
            -x_1+\i x_2&2\i v_2&x_5+\i x_6\\
            -x_3+\i x_4&-x_5+\i x_6&2\i v_3
           \end{pmatrix}.
\]
As before, we identify sections in a tensor bundle $E=G\times_H V$ over $M$ with $H$-equivariant functions on $G$ with values in $V$ and denote this by
\[\Gamma(E)\ni\varphi\mapsto\hat\varphi\in C^\infty(G,V)^H.\]
As seen in Remark~\ref{flagcs}, the Kähler form $\omega\in\Omega^2$ corresponds to the (constant) function
\[\hat\omega=e_{12}-e_{34}+e_{56}\in C^\infty(G,\Lambda^2\m)^H\]
(we write $e_{ij}=e_i\wedge e_j$ to shorten notation). Define a real-valued function $\hat\varphi\in C^\infty(G,\Lambda^2\m)$ by
\[\hat\varphi_\xi=v_1e_{56}-v_2e_{34}+v_3e_{12}.\]
Using the description of the Kähler form via $\hat\omega$ and the fact that $v_1+v_2+v_3=0$, it is easy to check that $\hat\varphi_\xi$ is in fact $\Lambda^{1,1}_0$-valued.

The functions $v_i\in C^\infty(G)$ are $H$-invariant since
\[v_i(x)=\langle\xi^\ast(x),h_i\rangle=\langle\Ad(x^{-1})\xi,h_i\rangle=\langle\xi,\Ad(x)h_i\rangle\]
and $\ad(h_j)h_i=[h_j,h_i]=0$, hence $dv_i(h_j)=0$. Using the commutator relations of $\u(3)$, one can check that the $2$-forms $e_{12},e_{34},e_{56}\in\Lambda^2\m$ are $H$-invariant as well. This implies that the function $\hat\varphi_\xi$ is $H$-equivariant. In total, $\hat\varphi_\xi\in C^\infty(G,\Lambda^{1,1}_{0,\R}\m)^H$ and thus $\hat\varphi_\xi$ projects to a primitive $(1,1)$-form $\varphi_\xi$ on $M$. The coclosedness of $\varphi_\xi$ has also been checked in \cite[Sec.~6]{hermitianlaplace}.

It is worth noting that in the language of harmonic analysis on homogeneous spaces, $\varphi_\xi$ is associated to the element
\[\xi\otimes F\in\su(3)^\C\otimes\Hom_{T^2}(\su(3)^\C,\Lambda^{1,1}_0\m),\]
where the Fourier coefficient $F$ is given by
\[F(X)=\langle X,h_1\rangle e_{56}-\langle X,h_2\rangle e_{34}+\langle X,h_3\rangle e_{12}.\]
It is therefore no surprise that $\bar\Delta\varphi_\xi=12\varphi_\xi$, since $12$ is the eigenvalue of the Casimir operator on $V_{(1,1)}=\su(3)^\C$ (see Table \ref{CasSU3}). The fact that each tensor $\varphi_\xi$ thus obtained is coclosed amounts to $\delta(F)=0$, where $\delta$ also denotes the prototypical differential operator
\[\delta: \Hom_{T^2}(\su(3),\Lambda^{1,1}_{0,\R}\m)\longrightarrow\Hom_{T^2}(\su(3),\m)\]
associated to the invariant differential operator $\delta: \Omega^{1,1}_{0,\R}\to\Omega^1$.

The corresponding symmetric $2$-tensor, which is the actual infinitesimal Einstein deformation, is now given as $h_\xi=-J\circ\varphi_\xi$. By composing $\hat\varphi_\xi$ with $\hat\omega_\xi$, we obtain
\[\hat h_\xi=v_3\cdot(e_1\otimes e_1+e_2\otimes e_2)+v_2\cdot(e_3\otimes e_3+e_4\otimes e_4)+v_1\cdot(e_5\otimes e_5+e_6\otimes e_6).\]
In this way, each $\xi\in\su(3)$ determines a unique element $\varphi_\xi\in\Omega^{1,1}_{0,\R}$ and hence a unique $h_\xi\in\varepsilon(g)$.

In passing, we note that the infinitesimal Einstein deformations of $F_{1,2}$ in fact coincide with the infinitesimal deformations of the nearly Kähler structure \cite[Cor.~5.12]{hermitianlaplace}. Their nonintegrability in the nearly Kähler sense was already established \cite{foscolo}. We now turn to the question whether the integrability in the Einstein sense is also obstructed.

\subsection{The obstruction against integrability}
\label{sec:flagobstr}

We first note that via the equivalence $\varepsilon(g)\cong\su(3)$ constructed in Section \ref{sec:flagied}, the integrability obstruction to second order
\[\mathcal{I}:\ \varepsilon(g)\times\varepsilon(g)\times\varepsilon(g)\to\R:\ \mathcal{I}(h_1,h_2,h_3):=\left(\Eop_g''(h_1,h_2),h_3\right)_{L^2}\]
can be viewed as a $G$-equivariant multilinear map that is symmetric in the first two entries, i.e. $\mathcal{I}\in(\Sym^2\su(3)^\ast\otimes\su(3)^\ast)^G$. Both of the spaces
\[\Sym^3\su(3)^G\subset(\Sym^2\su(3)\otimes\su(3))^G\]
turn out to be one-dimensional and hence equal -- in particular, $\mathcal{I}$ must be totally symmetric. Hence $\left(\Eop_g''(h,h),k\right)_{L^2}$ can be recovered from expressions of the type $\mathcal{I}(h,h,h)$ via polarization. Concretely,
\begin{equation}
\left(\Eop_g''(h,h),k\right)_{L^2}=\frac{1}{3}\frac{d}{dt}\big|_{t=0}\mathcal{I}(h+tk,h+tk,h+tk)\label{polarize}
\end{equation}
for $h,k\in\varepsilon(g)$.

The space $\Sym^3\su(3)^G$ is generated by the $G$-invariant cubic homogeneous polynomial $\i\det$. We therefore know that $\mathcal{I}(h_\xi,h_\xi,h_\xi)=c\cdot\i\det(\xi)$ for some $c\in\R$. Next, we proceed to show that $c\neq0$.

Introducing the notation $\alpha^\sigma(X_1,\ldots,X_r):=\alpha(X_{\sigma(1)},\ldots,X_{\sigma(r)})$ for any permutation $\sigma\in\Sl_r$ and any tensor $\alpha$ of rank $r$, we can rewrite formula (\ref{obstruction}) as
\[2\mathcal{I}(h,h,h)=\int_M2\Einstein\tr_g(h^3)\vol_g+3\left(\nabla^2h,h\otimes h\right)_{L^2}-6\left(\nabla^2h,(h\otimes h)^{(23)}\right)_{L^2}.\]
Integrating by parts and computing
\begin{align*}
\nabla^\ast(h\otimes h)&=-\sum_if_i\lrcorner\nabla_{f_i}(h\otimes h)=-\sum_if_i\lrcorner(\nabla_{f_i}h\otimes h+h\otimes\nabla_{f_i}h)\\
&=\delta h\otimes h-\sum_jf_j\otimes\nabla_{h(f_j)}h=-\nabla_{h(\cdot)}h,\\
\nabla^\ast(h\otimes h)^{(23)}&=-\sum_if_i\lrcorner(\nabla_{f_i}h\otimes h+h\otimes\nabla_{f_i}h)^{(23)}\\
&=(\delta h\otimes h)^{(12)}-\sum_j(f_j\otimes\nabla_{h(f_j)}h)^{(12)}=-(\nabla_{h(\cdot)}h)^{(12)}\\
\end{align*}
with some local orthonormal frame $(f_i)$ of $TM$, we obtain
\[\mathcal{I}(h,h,h)=\frac{1}{2}\int_M(2\Einstein I_0-3I_1+6I_2)\vol_g\]
for $h\in\varepsilon(g)$, where $I_0,I_1,I_2\in C^\infty(M)$ are defined by
\begin{align*}
I_0:=\tr_g(h^3),\quad I_1:=\langle\nabla h,\nabla_{h(\cdot)}h\rangle_g,\quad I_2:=\langle\nabla h,(\nabla_{h(\cdot)}h)^{(12)}\rangle_g.
\end{align*}

The functions $I_0,I_1,I_2$ on $M$ give rise to $H$-invariant functions $\hat I_0,\hat I_1,\hat I_2\in C^\infty(G)^H$, the first of which can already be easily computed:
\[\hat I_0=\tr(\hat h^3)=2v_1^3+2v_2^3+2v_3^3=6v_1v_2v_3,\]
using that $v_1+v_2+v_3=0$. In order to obtain the other two terms, we have to compute derivatives of $h$. Recall that the canonical Hermitian connection $\bar\nabla$ and the Levi-Civita connection $\nabla$ are related by
\[\nabla_X=\bar\nabla_X+\frac{1}{2}A_X,\qquad X\in TM,\]
where $A_X=J\circ(\nabla_XJ)$ on $TM$ and then extended as a derivation to tensors of arbitrary rank. Identifying $2$-forms with skew-symmetric endomorphisms of $TM$, we can also write $A_X=X\lrcorner\Psi^-$, where $\Psi^-\in\Omega^3$ is the imaginary part of the complex volume form of $M$, which is $G$-invariant and at the base point given by
\[\Psi^-=e_{236}-e_{146}-e_{135}-e_{245}.\]
(see also \cite[Sec.~6]{hermitianlaplace}).

The canonical horizontal distribution $\H\subset TG$ is spanned by the left-invariant vector fields $e_1,\ldots,e_6$. For any vector $X\in TM$, let $\tilde{X}\in\H$ denote its horizontal lift. Since $\bar\nabla$ is the Ambrose-Singer connection of the homogeneous space $M=G/H$, it follows from (\ref{CR}) that
\begin{align*}
\widehat{\bar\nabla_Xh}=\tilde{X}(\hat h)=&\,\tilde{X}(v_3)\cdot(e_1\otimes e_1+e_2\otimes e_2)+\tilde{X}(v_2)\cdot(e_3\otimes e_3+e_4\otimes e_4)\\
&+\tilde{X}(v_1)\cdot(e_5\otimes e_5+e_6\otimes e_6).
\end{align*}
We compute
\begin{align*}
e_1(\hat h)&=x_2\cdot(e_3\otimes e_3+e_4\otimes e_4-e_5\otimes e_5-e_6\otimes e_6),\\
e_2(\hat h)&=x_1\cdot(-e_3\otimes e_3-e_4\otimes e_4+e_5\otimes e_5+e_6\otimes e_6),\\
e_3(\hat h)&=x_4\cdot(e_1\otimes e_1+e_2\otimes e_2-e_5\otimes e_5-e_6\otimes e_6),\\
e_4(\hat h)&=x_3\cdot(-e_1\otimes e_1-e_2\otimes e_2+e_5\otimes e_5+e_6\otimes e_6),\\
e_5(\hat h)&=x_6\cdot(e_1\otimes e_1+e_2\otimes e_2-e_3\otimes e_3-e_4\otimes e_4),\\
e_6(\hat h)&=x_5\cdot(-e_1\otimes e_1-e_2\otimes e_2+e_3\otimes e_3+e_4\otimes e_4).
\end{align*}
Secondly, it follows from the $G$-invariance of $A$ that\footnote{We write $\alpha\odot\beta=\alpha\otimes\beta+\beta\otimes\alpha$ for $\alpha,\beta\in\m$.}
\begin{align*}
\widehat{A_Xh}=A_{\hat X}\hat h=&\,v_3\cdot(A_{\hat X}e_1\odot e_1+A_{\hat X}e_2\odot e_2)+v_2\cdot(A_{\hat X}e_3\odot e_3+A_{\hat X}e_4\odot e_4)\\
&+v_1\cdot(A_{\hat X}e_5\odot e_5+A_{\hat X}e_6\odot e_6).
\end{align*}
Using the above expression for $\Psi^-$, we compute
\begin{align*}
A_{e_1}\hat h&=(v_1-v_2)\cdot(e_3\odot e_5+e_4\odot e_6),\\
A_{e_2}\hat h&=(v_1-v_2)\cdot(e_4\odot e_5-e_3\odot e_6),\\
A_{e_3}\hat h&=(v_3-v_1)\cdot(e_1\odot e_5-e_2\odot e_6),\\
A_{e_4}\hat h&=(v_3-v_1)\cdot(e_1\odot e_6+e_2\odot e_5),\\
A_{e_5}\hat h&=(v_2-v_3)\cdot(e_1\odot e_3+e_2\odot e_4),\\
A_{e_6}\hat h&=(v_2-v_3)\cdot(e_1\odot e_4-e_2\odot e_3).
\end{align*}
To obtain $\nabla h$, we simply combine:
\[\hat{X}\lrcorner\widehat{\nabla h}=\widehat{\nabla_Xh}=\widehat{\bar\nabla_Xh}+\frac{1}{2}\widehat{A_Xh}=\tilde{X}(\hat h)+\frac{1}{2}A_{\hat X}\hat h.\]
The coefficients of $\widehat{\nabla h}\in C^\infty(G,\m^{\otimes 3})$ with respect to the basis $(e_i)$ are listed in Table \ref{nablahcoef}.

\begin{table}[t]
\centering
\begin{tabular}{c||c|c|c}
$i$&$\widehat{\nabla h}(e_i,e_1,\cdot)$&$\widehat{\nabla h}(e_i,e_2,\cdot)$&$\widehat{\nabla h}(e_i,e_3,\cdot)$\\\hline
$1$&$0$&$0$&$x_2e_3+\frac{v_1-v_2}{2}e_5$\\
$2$&$0$&$0$&$-x_1e_3+\frac{v_2-v_1}{2}e_6$\\
$3$&$x_4e_1+\frac{v_3-v_1}{2}e_5$&$x_4e_2+\frac{v_1-v_3}{2}e_6$&$0$\\
$4$&$-x_3e_1+\frac{v_3-v_1}{2}e_6$&$-x_3e_2+\frac{v_3-v_1}{2}e_5$&$0$\\
$5$&$x_6e_1+\frac{v_2-v_3}{2}e_3$&$x_6e_2+\frac{v_2-v_3}{2}e_4$&$-x_6e_3+\frac{v_2-v_3}{2}e_1$\\
$6$&$-x_5e_1+\frac{v_2-v_3}{2}e_4$&$-x_5e_2+\frac{v_3-v_2}{2}e_3$&$x_5e_3+\frac{v_3-v_2}{2}e_2$\\
\multicolumn{4}{c}{}\\
$i$&$\widehat{\nabla h}(e_i,e_4,\cdot)$&$\widehat{\nabla h}(e_i,e_5,\cdot)$&$\widehat{\nabla h}(e_i,e_6,\cdot)$\\\hline
$1$&$x_2e_4+\frac{v_1-v_2}{2}e_6$&$-x_2e_5+\frac{v_1-v_2}{2}e_3$&$-x_2e_6+\frac{v_1-v_2}{2}e_4$\\
$2$&$-x_1e_4+\frac{v_1-v_2}{2}e_5$&$x_1e_5+\frac{v_1-v_2}{2}e_4$&$x_1e_6+\frac{v_2-v_1}{2}e_3$\\
$3$&$0$&$-x_4e_5+\frac{v_3-v_1}{2}e_1$&$-x_4e_6+\frac{v_1-v_3}{2}e_2$\\
$4$&$0$&$x_3e_5+\frac{v_3-v_1}{2}e_2$&$x_3e_6+\frac{v_3-v_1}{2}e_1$\\
$5$&$-x_6e_4+\frac{v_2-v_3}{2}e_2$&$0$&$0$\\
$6$&$x_5e_4+\frac{v_2-v_3}{2}e_1$&$0$&$0$\\
\end{tabular}
\caption{Coefficients of $\widehat{\nabla h}$.}
\label{nablahcoef}
\end{table}

Now, we can finally tackle the terms $I_1$ and $I_2$ in the integrability obstruction. Let $(f_i)$ be a local orthonormal frame of $TM$. Then
\begin{align*}
I_1=\langle\nabla h,\nabla_{h(\cdot)}h\rangle_{T^\ast M^{\otimes 3}}&=\sum_i\langle\nabla_{f_i}h,\nabla_{h(f_i)}h\rangle_{T^\ast M^{\otimes 2}}\\
&=\sum_{i,j}h(f_i,f_j)\langle\nabla_{f_i}h,\nabla_{f_j}h\rangle_{T^\ast M^{\otimes 2}}.
\end{align*}
By the $G$-invariance of the Riemannian metric on $M$, it follows that
\begin{align*}
\hat I_1=\sum_{i,j}\widehat{h(f_i,f_j)}\langle\widehat{\nabla_{f_i}h},\widehat{\nabla_{f_j}h}\rangle_{\m^{\otimes 2}}&=\sum_{i,j}\hat h(\hat f_i, \hat f_j)\langle\hat f_i\lrcorner\widehat{\nabla h},\hat f_j\lrcorner\widehat{\nabla h}\rangle_{\m^{\otimes 2}}.
\end{align*}
Note that $(\hat f_i(x))$ forms an orthonormal basis of $\m$ at each point $x\in G$. Since the above expression is independent of the choice of orthonormal basis, we can substitute in the orthonormal basis $(e_i)$ of $\m$. Hence the above is equal to
\[\hat I_1=\sum_{i,j}\hat h(e_i,e_j)\langle e_i\lrcorner\widehat{\nabla h},e_j\lrcorner\widehat{\nabla h}\rangle_{\m^{\otimes 2}}.\]
Similarly, we have
\[\hat I_2=\sum_{i,j}\hat h(e_i,e_j)\langle e_i\lrcorner(\widehat{\nabla h})^{(12)},e_j\lrcorner\widehat{\nabla h}\rangle_{\m^{\otimes 2}}.\]
Plugging in the coefficients from Table \ref{nablahcoef}, we obtain
\begin{align*}
\hat I_1&=-18v_1v_2v_3+4(x_1^2+x_2^2)v_3+4(x_3^2+x_4^2)v_2+4(x_5^2+x_6^2)v_1,\\
\hat I_2&=9v_1v_2v_3.
\end{align*}
One can check that these functions are indeed $H$-invariant and thus project to functions $I_0,I_1,I_2$ on $M$. Recall that $\scal_g=30$, whence $\Einstein=5$. Subsuming the above results, the full integrability obstruction is given by
\begin{align*}
\mathcal{I}(h,h,h)&=\frac{1}{2}\int_M(10I_0-3I_1+6I_2)\vol=\frac{1}{\Vol(K)}\int_GI\vol,\\
I&=84v_1v_2v_3-6(x_1^2+x_2^2)v_3-6(x_3^2+x_4^2)v_2-6(x_5^2+x_6^2)v_1.\\
\end{align*}
Integrating over $G$ amounts to projecting the integrand to its $G$-invariant, i.e. constant, part. If we view $v_1,\ldots,x_6$ as linear forms on $\su(3)$, the integrand $I$
is an $H$-invariant cubic homogeneous polynomial in $\su(3)^\ast$. Recall that the inner product on $\su(3)$ is induced by $-\frac{1}{12}B$, and
\[v_i=\langle\xi^\ast,h_i\rangle,\quad x_i=\langle\xi^\ast, e_i\rangle.\]
We therefore have the relations $\langle x_i,x_j\rangle=\langle e_i,e_j\rangle=\delta_{ij}$ as well as $\langle x_i,v_j\rangle=\langle e_i,h_j\rangle=0$, while
\[\langle v_i,v_j\rangle=\langle\pr_{\su(3)}h_i,\pr_{\su(3)}h_j\rangle=\begin{cases}
                                                                        \frac{1}{3}&i=j,\\
                                                                        -\frac{1}{6}&i\neq j.
                                                                       \end{cases}\]
The generator $\i\det$ of $\Sym^3\su(3)^G$ can be written as
\begin{align*}
\i\det=&\,8v_1v_2v_3+2(x_1x_3x_5-x_1x_4x_6-x_2x_3x_6-x_2x_4x_5)\\
&-2(x_1^2+x_2^2)v_3-2(x_3^2+x_4^2)v_2-2(x_5^2+x_6^2)v_1.
\end{align*}
The inner product on $\Sym^k\su(3)$ is induced by the inner product on $\su(3)$ via
\[\langle a_1\cdots a_k,b_1\cdots b_k\rangle=\sum_{\sigma\in\Sl_k}\prod_{i=1}^k\langle a_i,b_{\sigma(i)}\rangle\quad\text{ for }a_1,\ldots,a_k,b_1,\ldots,b_k\in\su(3).\]
We therefore see that
\begin{align*}
\langle I,\i\det\rangle_{\Sym^3\su(3)}=&\,84\cdot3\cdot|v_1v_2v_3|^2_{\Sym^3\su(3)}\\
&+6\cdot2\cdot(|x_1^2v_3|^2_{\Sym^3\su(3)}+\ldots+|x_6^2v_1|^2_{\Sym^3\su(3)})\\
=&\,672\cdot\frac{1}{18}+12\cdot6\cdot\frac{2}{3}=\frac{256}{3}\neq0
\end{align*}
and hence $\mathcal{I}(h,h,h)=c\cdot\i\det(h)$ for some $c\neq0$.

Suppose now that $\det(\xi)=0$ for some nonzero $\xi\in\su(3)$. By equation (\ref{polarize}), $\Eop_g''(h_\xi,h_\xi)$ is orthogonal to $\varepsilon(g)$ if and only if $\xi$ is a critical point of $\det$, i.e. if
\[\frac{d}{dt}\big|_{t=0}\det(\xi+t\eta)=0\]
for all $\eta\in\su(3)$. Equivalently, the rank of the complex $3\times3$-matrix $\xi$ is equal to $1$. However, no such element of $\su(3)$ exists, since nonzero skew-hermitian matrices have rank at least $2$.

This concludes the proof of Theorem~\ref{thm2}.

\end{document}